\documentclass[a4paper,10pt,psamsfonts]{amsart}

\usepackage[utf8]{inputenc} 
\usepackage[T1]{fontenc}    

\usepackage{amsmath}
\usepackage{amstext}
\usepackage{amsfonts}
\usepackage{amsthm}
\usepackage{amssymb}

\usepackage{graphicx}
\usepackage{psfrag}
\usepackage{float}

\usepackage{hyperref}
\hypersetup{colorlinks}
\hypersetup{linktocpage}

\usepackage{color}

\newcommand{\R}{{\mathbb{R}}}
\newcommand{\E}{{\mathbb{E}}}
\newcommand{\N}{{\mathbb{N}}}

\newcommand{\F}{{\mathcal{F}}} 
\renewcommand{\P}{{\mathbf{P}}} 

\newcommand{\B}{{\mathcal{B}}}

\newcommand{\diff}[1]{\,\mathrm{d}#1}

\newcommand{\one}{\mathbb{I}}
\newcommand{\id}{\mathrm{id}}
\newcommand{\ee}{\mathrm{e}}

\theoremstyle{plain}

\newtheorem{definition}{Definition}[section]
\newtheorem{theorem}[definition]{Theorem}
\newtheorem{lemma}[definition]{Lemma}
\newtheorem{corollary}[definition]{Corollary}
\newtheorem{prop}[definition]{Proposition}
\newtheorem{assumption}[definition]{Assumption}

\theoremstyle{definition}
\newtheorem{remark}[definition]{Remark}

\newtheorem{example}[definition]{Example}

\begin{document}

\title[Stoch. C-Stab. and B-cons. of Euler-type
schemes] 
{Stochastic C-stability and B-consistency of explicit and implicit Euler-type
schemes }

\author[W.-J.~Beyn]{Wolf-J\"urgen Beyn}
\address{Wolf-J\"urgen Beyn\\
Fakult\"at f\"ur Mathematik\\
Universit\"at Bielefeld\\
Postfach 100 131\\
DE-33501 Bielefeld\\
Germany}
\email{beyn@math.uni-bielefeld.de}

\author[E.~Isaak]{Elena Isaak}
\address{Elena Isaak\\
Fakult\"at f\"ur Mathematik\\
Universit\"at Bielefeld\\
Postfach 100 131\\
DE-33501 Bielefeld\\
Germany}
\email{eisaak@math.uni-bielefeld.de}

\author[R.~Kruse]{Raphael Kruse}
\address{Raphael Kruse\\
Technische Universit\"at Berlin\\
Institut f\"ur Mathematik, Secr. MA 5-3\\
Stra\ss e des 17.~Juni 136\\
DE-10623 Berlin\\
Germany}
\email{kruse@math.tu-berlin.de}


\begin{abstract}
  This paper is concerned with the numerical approximation of stochastic
  ordinary differential equations, which satisfy a global monotonicity
  condition. This condition includes several equations with super-linearly
  growing drift and diffusion coefficient functions such as
  the stochastic Ginzburg-Landau equation and the 3/2-volatility model from
  mathematical finance. Our analysis of the mean-square error of convergence
  is based on a suitable generalization of the notions of C-stability and
  B-consistency known from deterministic numerical analysis for stiff ordinary
  differential equations. An important feature of our stability concept is that
  it does not rely on the availability of higher moment bounds of the 
  numerical one-step scheme.  
  
  While the convergence theorem is derived in a somewhat more abstract
  framework, this paper also contains two more concrete examples of
  stochastically C-stable numerical one-step schemes: the split-step backward
  Euler method from Higham et al.\ (2002) and a newly proposed explicit variant
  of the Euler-Maruyama scheme, the so called projected Euler-Maruyama method.
  For both methods the optimal rate of strong convergence is proven
  theoretically and verified in a series of numerical experiments.
\end{abstract}

\maketitle

\section{Introduction}
\label{sec:intro}
Initiated by the papers \cite{higham2002b} and \cite{hu1996}
the field of numerical analysis for stochastic ordinary differential equations
(SODEs) with super-linearly growing coefficient functions has seen a 
considerable progress, especially over the last couple of years. For instance,
we refer to \cite{ hutzenthaler2014a, hutzenthaler2014c, hutzenthaler2012,
mao2013a, sabanis2013b, tretyakov2013} and the references therein.

The starting point of this article is the following observation: There exist
strongly convergent numerical schemes, whose one-step maps
satisfy suitable Lipschitz-type conditions, although the
underlying stochastic differential equation has non-globally Lipschitz
continuous coefficient functions. For
the numerical approximation of stiff deterministic ODEs this observation has
been formalized in the notion of \emph{C-stability}, see for example
\cite[Definition 2.1.3]{dekker1984} and \cite[Chap.~8.4]{strehmel2012}. A
related result is also found in \cite[Prop.~15.2]{hairer1996}.

In this paper we present a generalization of this notion to the stochastic
situation. Together with its counterpart, the notion of \emph{B-consistency},
we will show that the error analysis of stochastically C-stable
numerical methods can be simplified significantly compared to existing
approaches in the literature. In particular, it turns out
that it is not necessary to study higher moment estimates of the numerical
scheme nor to consider their continuous time extensions.

We apply this more abstract framework to study the strong error of convergence
for the numerical discretization of SODEs under the
\emph{global monotonicity condition} (see \eqref{eq3:onesided}). This condition
includes several examples of SODEs with superlinearly growing drift and
diffusion coefficient functions, for which the explicit Euler-Maruyama method
is known to be divergent, see \cite{hutzenthaler2011}. However, several
explicit and implicit variants of the Euler-Maruyama method have been developed
and analyzed in recent papers on this topic.
For instance, we
refer to \cite{mao2013a} for the strong error analysis of the backward Euler
method, and to \cite{sabanis2013b, tretyakov2013} for a corresponding result of
the explicit tamed Euler method. Further, in \cite{hutzenthaler2014c} strong
convergence rates are derived for a stopped-tamed Euler-Maruyama method applied
to SODEs which lie beyond the global monotonicity condition.

In this paper we work with the following notion of strong convergence:
We say that a numerical scheme converges strongly with order
$\gamma$ to the exact solution $X \colon [0,T] \times \Omega \to \R^d$
if there exists a constant $C$ independent of the temporal step size $h$ such
that 
\begin{align}
  \label{eq0:strerr}
  \max_{n \in \{1,\ldots,N\} } \| X(t_n) - X_h(t_n) \|_{L^2(\Omega;\R^d)} \le C
  |h|^\gamma.
\end{align}
Here, $X_h \colon \{t_0, t_1, \ldots,t_N\} \times \Omega
\to \R^d$ denotes the grid function generated by the numerical scheme. Let us
remark that several of the above mentioned papers consider stronger
notions of strong convergence, where, for example, the maximum occurs 
inside the $L^2$-norm or the norm in $L^p(\Omega;\R^d)$ with $p > 2$ is
considered instead of the $L^2$-norm. Our choice of \eqref{eq0:strerr} is
explained by the fact that our proof of the stability lemma (see
Lemma~\ref{lem:stab}), which plays a central role in our approach, relies on
the orthogonality of the conditional expectation with respect to the norm in
$L^2$. 

In order to demonstrate the usefulness of our abstract results we present
two more concrete examples of stochastically C-stable numerical schemes:
First we are concerned with the \emph{split-step backward Euler method} (SSBE)
from \cite{higham2002b}, which is shown to be strongly convergent of order
$\gamma = \frac{1}{2}$ in Theorem~\ref{cor:SSBEconv}. Second, we propose
a new explicit scheme, the \emph{projected Euler-Maruyama method} (PEM),
which turns out to be, in general, computationally less expensive then the
implicit SSBE scheme. In Theorem~\ref{th:PEMconv} we verify that the PEM method
is also strongly convergent of order $\frac{1}{2}$.

In our numerical experiments in Section~\ref{sec:exp} both methods perform
equally well in terms of the experimental strong errors, which therefore
indicates to favour the explicit PEM method due to its simpler implementation.
However, this only holds true in the non-stiff case. As for
deterministic ODEs, stiff problems may require an impractical small step size
for an explicit numerical method while implicit schemes already give more
useful results for larger step sizes, therefore reducing the overall
computational cost. This is relevant if, for instance, the numerical one-step 
method is used for the time integration of a parabolic stochastic partial
differential equation. Although we apply techniques from the numerical analysis
of stiff equations in our error analyis we leave this issue to future research
and concentrate here on non-stiff problems. In this context we also refer to
\cite{hutzenthaler2012} for a detailed comparison between implicit numerical
methods and a further purely explicit variant of the Euler-Maruyama method, the
tamed Euler method, which is considered in several of the above mentioned
papers. 

Let us briefly highlight two results in the literature, which are
closely related to our approach from a methodological point of view: In
\cite{wang2012} the authors investigate a family of one-leg theta methods for
the discretization of SODEs under a one-sided Lipschitz condition on the drift
and a global Lipschitz bound on the diffusion coefficient function. Hereby,
they make use of the related notion of B-convergence. The second paper
\cite{tretyakov2013} presents a fundamental mean square convergence theorem for
the discretization of SODEs under the global monotonicity condition. This
theorem imposes a similar concept of the local truncation error as our notion
of B-consistency. However, in the proof of the theorem the authors relate the
global error at time $t_i$ to the error at time $t_{i-1}$ by one time step of
the exact solution. Proceeding in this way one cannot benefit from the global
Lipschitz properties of the numerical method.

The remainder of this paper is organized as follows: The following section
contains a detailed description of the class of stochastic ordinary differential
equations, whose solutions we want to approximate. Further, we state our main
assumptions and present the numerical schemes, which are analyzed in the
subsequent sections. In Section~\ref{sec:def} we develop 
our notions of stochastic C-stability and B-consistency in a somewhat more
abstract framework. Then we prove the already mentioned stability lemma, from
which we easily deduce our strong convergence theorem for C-stable numerical
methods. 

In Section~\ref{sec:nonlinear} we briefly summarize some results on the
solvability of nonlinear equations, which are needed for the error analysis of
the SSBE method. In Sections~\ref{sec:SSBE} and \ref{sec:PEM} we verify
that the split-step backward
Euler scheme and the projected Euler-Maruyama method are stochastically
C-stable and B-consistent, and, hence, strongly convergent. 
In Section~\ref{sec:exp} we present some numerical experiments which illustrate
our theoretical results for the discretization of the stochastic
Ginzburg-Landau equation and for the financial $3/2$-volatility model.

\section{Problem description and the numerical methods}
\label{sec:prob}
In this section we introduce the class of stochastic differential equations,
which we aim to discretize. Further, we state our main assumptions and
the numerical methods, which we study in the remainder of this paper.

Let $d,m \in \N$, $T \in (0,\infty)$, and $(\Omega, \F, (\F_t)_{t \in [0,T]},
\P)$ be a filtered probability space satisfying the usual conditions. We
consider the solution $X \colon [0,T] \times \Omega \to \R^d$ to
the SODE
\begin{align}
  \label{sode}
  \begin{split}
    \diff{X(t)} &= f(t,X(t)) \diff{t} + \sum_{r=1}^m g^r(t,X(t))
    \diff{W^r(t)},\quad t \in [0,T],  \\
    X(0)&=X_0.
  \end{split}
\end{align}
Here $f\colon [0,T] \times \R^d \to \R^d$ stands for the drift coefficient
function, while $g^r \colon [0,T]\times \R^d \to \R^d$, $r=1,\ldots,m$,
are the diffusion coefficient functions. By $W^r \colon [0,T] \times
\Omega \to \R$, $r = 1,\ldots,m$, we denote an independent family of
real-valued standard $(\F_t)_{t\in [0,T]}$-Brownian motions on
$(\Omega,\mathcal{F},\P)$. For a sufficiently large $p \in [2,\infty)$ the
initial condition $X_0$ is assumed to be an element of the space
$L^p(\Omega,\F_0,\P;\R^d)$.

By $\langle \cdot, \cdot \rangle$ and $|\cdot|$ we denote the Euclidean inner
product and the Euclidean norm on $\R^d$, respectively. Throughout this paper
we impose the following conditions on the drift and the diffusion coefficient
functions. Note that the strong convergence result for the SSBE method in
Theorem~\ref{cor:SSBEconv} requires a more restrictive lower bound for the
parameter $\eta$ appearing in \eqref{eq3:onesided}.

\begin{assumption}
  \label{as:fg}
  The mappings $f \colon [0,T] \times \R^d \to \R^d$ and
  $g^r \colon [0,T] \times \R^d \to \R^d$, $r = 1,\ldots,m$, are
  continuous. Furthermore, there exist a positive constant $L$ and a parameter
  value $\eta \in (\frac{1}{2},\infty)$ with
  \begin{align}
    \label{eq3:onesided}
    \big\langle f(t,x_1) - f(t,x_2), x_1-x_2 \big\rangle + \eta \sum_{r =
    1}^m  \big| g^r(t,x_1) - g^r(t,x_2) \big|^2 &\le L | x_1 - x_2 |^2
  \end{align}
  for all $t \in [0,T]$ and $x_1,x_2 \in \R^d$.
  In addition, there exists a constant $q \in (1,\infty)$ such that for every
  $r = 1,\ldots,m$ it holds
  \begin{align}
    |f(t,x) | \vee | g^r(t,x)| &\le L \big( 1 + |x |^q \big),
    \label{eq3:poly_growth}\\
    | f(t_1,x) - f(t_2,x) | \vee | g^r(t_1,x) - g^r(t_2,x) | &\le L \big( 1 +
    |x|^q \big)  |t_1 - t_2|^{\frac{1}{2}},\label{eq3:loc_Lip_t}\\
    | f(t,x_1) - f(t,x_2) |\vee | g^r(t,x_1) - g^r(t,x_2) | &\le L \big( 1 +
    |x_1|^{q-1} + |x_2 |^{q-1} \big)  | x_1 - x_2 |,
    \label{eq3:loc_Lip}
  \end{align}
  for all $t,t_1,t_2 \in [0,T]$ and $x,x_1,x_2 \in \R^d$.
\end{assumption}

The assumption \eqref{eq3:onesided} is called \emph{global monotonicity
condition}. We exclude the case $q = 1$, since this coincides with the
well-known global Lipschitz case studied in \cite{kloeden1999,milstein1995}. In
Section~\ref{sec:exp} we present two more concrete SODEs, which fulfill 
Assumption~\ref{as:fg}.

Before we describe the numerical schemes we remark that Assumption~\ref{as:fg}
is also sufficient to ensure the existence of a unique solution to
\eqref{sode}, see \cite{krylov1999}, \cite[Chap.~2.3]{mao1997} or
\cite[Chap.~3]{roeckner2007}. By this we understand 
an almost surely continuous and $(\F_t)_{t \in [0,T]}$-adapted stochastic
process $X \colon [0,T] \times \Omega \to \R^d$ which satisfies $\P$-almost
surely the integral equation
\begin{align}
  \label{exact}
  X(t) = X_0 + \int_{0}^{t} f(s,X(s)) \diff{s} + \sum_{r = 1}^m \int_{0}^t
  g^r(s,X(s)) \diff{W^r(s)}
\end{align}
for all $t \in [0,T]$. In addition, if there exist $C \in (0,\infty)$ and $p
\in [2, \infty)$ such that
\begin{align}
  \label{eq:growthcond}
  \big\langle f(t,x), x \big\rangle + \frac{p-1}{2} \sum_{r =
  1}^m  \big| g^r(t,x) \big|^2 &\le C \big(1 + | x |^2 \big)
\end{align}
for all $x \in \R^d$, $t \in [0,T]$, then the exact solution has a finite $p$-th
moment, that is
\begin{align}
  \label{eq:moments}
  \sup_{t \in [0,T]} \big\| X(t) \big\|_{L^p(\Omega;\R^d)} < \infty.
\end{align}
For a proof we refer, for instance, to \cite[Chap.~2.4]{mao1997}. The condition
\eqref{eq:growthcond} is called \emph{global coercivity condition}.

For the formulation of the numerical methods we introduce the following
terminology: For $N \in \N$ we say that $h = (h_1,\ldots,h_{N}) \in (0,T]^{N}$
is a \emph{vector of (deterministic) step sizes} if $\sum_{i = 1}^N h_i = T$.
Every vector of step sizes $h$ gives rise to a set of temporal grid points
$\mathcal{T}_h$, which is given by
\begin{align*}
  \mathcal{T}_h := \Big\{ t_n := \sum_{i = 1}^n h_i \, : \, n = 0,\ldots,N
  \Big\}.
\end{align*}
For short we write $|h| := \max_{i \in \{ 1,\ldots,N\}} h_i$
for the \emph{maximal step size} in $h$. 

The aim of this paper is to show that the following two schemes are 
examples of stochastically C-stable numerical methods.


\begin{example}
  \label{ex:SSBE1}
  Consider the so called \emph{split-step backward Euler method}
  (SSBE) studied in \cite{higham2002b}. For its formulation
  let $h = (h_1,\ldots,h_N)$ be a vector of step sizes. Then the SSBE method is
  given by setting $X_h^{\mathrm{SSBE}}(0) = X_0$ and by the recursion
  \begin{align} 
    \label{SSBErecursion}
    \begin{split}
      \overline{X}_h^{\mathrm{SSBE}}(t_i) &= X_h^{\mathrm{SSBE}}(t_{i-1}) + h_i
      f(t_{i}, \overline{X}_h^{\mathrm{SSBE}}(t_i)),\\
      X_h^{\mathrm{SSBE}}(t_{i}) &= \overline{X}_h^{\mathrm{SSBE}}(t_i) +
      \sum_{r = 1}^m g^r(t_{i}, \overline{X}_{h}^{\mathrm{SSBE}}(t_i)) \big(
      W^r(t_{i}) - W^r(t_{i-1}) \big),
    \end{split}
  \end{align}
  for every $i = 1,\ldots,N$. It is shown in Section~\ref{sec:SSBE} that the
  SSBE scheme is a well-defined stochastic one-step method under
  Assumption~\ref{as:fg} and that it is strongly convergent of order $\gamma =
  \frac{1}{2}$, see Theorem \ref{cor:SSBEconv}.

  Let us note that we evaluate the diffusion coefficient functions $g^r$ at
  time $t_{i}$ in the $i$-th step of the SSBE method. This
  is somewhat unusual when compared to the definition
  of the backward Euler scheme in \cite[Chap.~12]{kloeden1999}, where 
  $g^r$ is evaluated at $t_{i-1}$ instead.
  The reason for this slight modification lies in condition
  \eqref{eq3:onesided}, which is applied to $f$ and $g^r$, $r =
  1,\ldots,m$, simultaneously at the same point $t$ in time. Compare also with
  the inequality \eqref{eq:stab} further below. It helps to avoid some
  technical issues if we already take this relationship into consideration in
  the definition of the numerical scheme.
\end{example}

\begin{example}
  \label{ex:PEM}   
  Another example of a stochastically C-stable scheme is the following
  explicit variant of the Euler-Maruyama method, which we term \emph{projected
  Euler-Maruyama method} (PEM). It consists of the standard Euler-Maruyama
  method and a projection onto a ball in $\R^d$ whose radius is expanding with
  a negative power of the step size.

  To be more precise, let $h \in (0,1]^N$ be an arbitrary vector of step sizes.
  The parameter value $\alpha \in (0,1]$ is chosen to be $\alpha =
  \frac{1}{2(q-1)}$ in dependence on the growth rate $q$ appearing in
  Assumption~\ref{as:fg}. Then, the PEM method is given by the recursion
  \begin{align} 
    \label{PEMrecursion}
    \begin{split}
      \overline{X}_h^{\mathrm{PEM}}(t_i) &:= \min\big( 1 , h_i^{-\alpha} \big|
      X_h^{\mathrm{PEM}}(t_{i-1}) \big|^{-1} \big)
      X_h^{\mathrm{PEM}}(t_{i-1}),\\
      X_h^{\mathrm{PEM}}(t_i) &:= \overline{X}_h^{\mathrm{PEM}}(t_{i})
      + h_i f(t_{i-1}, \overline{X}_h^{\mathrm{PEM}}(t_{i}) )\\
      &\quad + \sum_{r = 1}^m g^r(t_{i-1}, \overline{X}_h^{\mathrm{PEM}}(t_{i}))
      \big( W^r(t_{i}) - W^r(t_{i-1}) \big), \; \text{ for } 1 \le i \le N,
    \end{split}
  \end{align}
  where $X_h^{\mathrm{PEM}}(0) := X_0$. The strong error
  analysis of the PEM method is carried out in Section~\ref{sec:PEM}.

  To the best of our knowledge the PEM method for stochastic equations is new
  to the literature. Its definition is inspired by a truncation procedure, which
  plays an important role in the proof of \cite[Chap.~2, Theorem~3.4]{mao1997}.
  For deterministic ODEs projection methods appear in geometric integration,
  see \cite{grimm2005}. After the first preprint of this paper has appeared on
  \texttt{arxiv.org} the asymptotic stability and integrability property of a
  variant of the PEM method \eqref{PEMrecursion} is studied in 
  \cite{szpruch2015} using Lyapunov function techniques.
\end{example}

\section{An abstract convergence theorem}
\label{sec:def}
This section contains a detailed introduction to our notions of stochastic
C-stability and B-consistency in a somewhat more abstract framework. Then we
state our strong convergence theorem, whose proof turns out to be a 
direct application of the stability Lemma~\ref{lem:stab}.

We begin by introducing some additional notation. By $\overline{h} \in (0,T]$
we denote an \emph{upper step size bound} and we define the set $\mathbb{T} :=
\mathbb{T}(\overline{h}) \subset [0,T) \times (0,\overline{h}]$ to be   
\begin{align*}
  \mathbb{T}  := \big\{ (t,\delta) \in [0,T) \times
  (0,\overline{h}] \, : \, t+\delta \le T \big\}. 
\end{align*}
Further, for a given vector of step sizes $h \in (0,\overline{h}]^N$ we denote
by $\mathcal{G}^{2}(\mathcal{T}_h)$ the space of all adapted and square
integrable \emph{grid functions}, that is
\begin{align*}
  \mathcal{G}^{2}(\mathcal{T}_h) := \big\{ Z \colon \mathcal{T}_h \times
  \Omega \to \R^d\, : \, Z(t_n) \in L^2(\Omega,\F_{t_n},\P;\R^d)\text{ for all
  } n = 0,1,\ldots,N \big\}.
\end{align*}
Our abstract class
of stochastic one-step methods is defined as follows. 

\begin{definition}
  \label{def:onestep}
  Let $\overline{h} \in (0,T]$ be an upper step size bound and $\Psi \colon
  \R^d \times \mathbb{T} \times \Omega \to \R^d$ be a mapping satisfying the
  following measurability and integrability condition: For every 
  $(t,\delta) \in \mathbb{T}$ and $Z \in L^2(\Omega,\F_{t},\P;\R^d)$ it holds
  \begin{align}
    \label{eq:Psicond}
    \Psi(Z,t,\delta) \in L^2(\Omega,\F_{t+\delta},\P;\R^d).
  \end{align}
  Then, for every vector of step sizes $h = (h_1,\ldots,h_N) \in
  (0,\overline{h}]^N$, $N \in \N$, we say that 
  a grid function $X_h \in \mathcal{G}^2(\mathcal{T}_h)$ is generated by the
  \emph{stochastic one-step method} $(\Psi,\overline{h},\xi)$ with initial
  condition $\xi \in L^2(\Omega,\F_{0},\P;\R^d)$ if 
  \begin{align}
    \label{eq:onestep2}
    \begin{split}
      X_h(t_i) &= \Psi(X_h(t_{i-1}), t_{i-1}, h_i), \quad 1 \le i \le N,\\
      X_h(t_0) &= \xi.
    \end{split}
  \end{align}
We call $\Psi$ the \emph{one-step map} of the method.
\end{definition}

Next, we present our definition of stability for stochastic one-step methods.
 It is a suitable generalization of the notion of
C-stability from \cite[Definition 2.1.3]{dekker1984} and has been used in the
context of numerical approximation of stiff differential equations. We also
refer to \cite[Prop.~15.2]{hairer1996} and to \cite[Chap.~8.4]{strehmel2012} for
a more recent exposition. 

\begin{definition}
  \label{def:cstab}
  A stochastic one-step method $(\Psi,\overline{h},\xi)$ is called
  \emph{stochastically C-stable} (with respect to the norm in
  $L^2(\Omega;\R^d)$) if there exist a constant $C_{\mathrm{stab}}$ and a
  parameter value $\eta \in (1,\infty)$ such that for all $(t,\delta) \in
  \mathbb{T}$ and all random variables $Y, Z \in L^2(\Omega,\F_{t},\P;\R^d)$
  it holds  
  \begin{align}
    \label{eq:stab_cond1}
    \begin{split}
      &\big\| \E \big[ \Psi(Y,t,\delta) -
      \Psi(Z,t,\delta) | \F_{t} \big]
      \big\|_{L^2(\Omega;\R^d)}^2\\
      &\qquad + \eta \big\| \big( \id - \E [ \, \cdot \, |
      \F_{t} ] \big) \big( \Psi(Y,t,\delta) -
      \Psi(Z,t,\delta) \big) \big\|^2_{L^2(\Omega;\R^d)}\\
      &\quad\le \big(1 + C_{\mathrm{stab}} \delta \big)
      \big\| Y - Z \big\|_{L^2(\Omega;\R^d)}^2.    
    \end{split}
  \end{align}
\end{definition}
Here and in what follows we denote by $\big(\id - \E [ \, \cdot \, |\F_{t}
]\big)Y = Y - \E[Y|\F_{t}]$ the projection of an $\R^d$-valued random variable
orthogonal to the conditional expectation $\E[Y| \F_{t}]$.
The next definition is concerned with the local truncation error. The
conditions \eqref{eq:cons_cond1} and \eqref{eq:cons_cond2} are well-known in
the literature and are found in slightly different form in
\cite[Th.~1.1]{milstein1995} and \cite[Th.~1.1]{milstein2004}. A related
concept has been applied in \cite{tretyakov2013}, but the authors need higher
moment estimates of the local truncation error. 

\begin{definition}
  \label{def:bcons}
  A stochastic one-step method $(\Psi,\overline{h},\xi)$ is called
  \emph{stochastically B-consistent} of order $\gamma > 0$ to \eqref{sode} if
  there exists a constant $C_{\mathrm{cons}}$  such that for every
  $(t,\delta) \in \mathbb{T}$ it holds
  \begin{align}
    \label{eq:cons_cond1}
    \begin{split}
      \big\| \E \big[ X(t+\delta) - \Psi(X(t),t,\delta) | \F_{t} \big] 
      \big\|_{L^2(\Omega;\R^d)} \le C_{\mathrm{cons}} \delta^{ \gamma + 1}
    \end{split}
  \end{align}
  and
  \begin{align}
    \label{eq:cons_cond2}
    \begin{split}
      \big\| \big( \id - \E [ \, \cdot \, | \F_{t} ] \big)
      \big( X(t + \delta) - \Psi(X(t),t,\delta) \big)
      \big\|_{L^2(\Omega;\R^d)} \le C_{\mathrm{cons}} \delta^{ \gamma +
      \frac{1}{2}},
    \end{split}
  \end{align}
  where $X \colon [0,T] \times \Omega \to \R^d$ denotes the exact solution to
  \eqref{sode}.
\end{definition}

Finally, it remains to give our definition of strong convergence.

\begin{definition}
  \label{def:conv}
  A stochastic one-step method $(\Psi,\overline{h},\xi)$
  \emph{converges} \emph{strongly} with order $\gamma > 0$ to the exact
  solution of \eqref{sode} if there exists a constant $C$ such that for every
  vector of step sizes $h \in (0,\overline{h}]^N$ it holds
  \begin{align*}
    \max_{n \in \{0,\ldots,N\}} \big\| X_h(t_n) - X(t_n)
    \big\|_{L^{2}(\Omega;\R^d)} \le C |h|^{\gamma}.
  \end{align*}
  Here $X$ denotes the exact solution to \eqref{sode} and $X_h \in
  \mathcal{G}^2(\mathcal{T}_h)$ is the grid function generated by
  $(\Psi,\overline{h},\xi)$ with step sizes $h \in (0,\overline{h}]^N$. 
\end{definition}

We first prove the following
useful stability lemma. It follows from the discrete Gronwall Lemma and gives a
motivation for the conditions \eqref{eq:stab_cond1} to \eqref{eq:cons_cond2}.
The underlying principle is similar to  the proof of
\cite[Th.~1.1]{milstein1995} and \cite[Th.~1.1]{milstein2004}, but differs in
one important point: In \cite[Th.~1.1]{milstein1995} the error at time $t_i$
is related to the error at time $t_{i-1}$ by one discrete time step of the
exact solution (compare with \cite[Lemma~1.1]{milstein1995}). Here we follow 
the same idea, but we propagate the error by one application of the one-step
map.
 This turns out to be important since a stochastically C-stable
one-step method enjoys a global Lipschitz property.

\begin{lemma}
  \label{lem:stab}
  Let $(\Psi,\overline{h},\xi)$ be a stochastically C-stable one-step method
  with constants $C_\mathrm{stab}$ and $\eta \in (1, \infty)$. Let $h
  \in (0,\overline{h}]^N$ be an arbitrary vector of step sizes. For
  every grid function $Z \in \mathcal{G}^2(\mathcal{T}_h)$ it then
  follows that
  \begin{align*}
    &\max_{n \in \{0,\ldots,N\}} \| Z(t_n) - X_h(t_n)  \|_{L^2(\Omega;\R^d)}^2 
    \le \ee^{(1 + C_{\mathrm{stab}}(1 + \overline{h}))T} \Big( \| Z(0) -
    \xi  \|_{L^2(\Omega;\R^d)}^2 \\ 
    &\qquad + \sum_{i = 1}^{N} \big(1 + h_i^{-1}\big) \big\| \E \big[
    Z(t_i) - \Psi(Z(t_{i-1}), t_{i-1}, h_i) | \F_{t_{i-1}} \big]
    \big\|^2_{L^2(\Omega;\R^d)}\\ 
    &\qquad + C_\eta \sum_{i = 1}^{N}  \big\| \big(
    \id - \E \big[ \, \cdot \, 
    | \F_{t_{i-1}} \big] \big) \big( Z(t_i) -
    \Psi(Z(t_{i-1}), t_{i-1}, h_i) \big) \big\|^2_{L^2(\Omega;\R^d)}  \Big),
  \end{align*}
  where $C_\eta = 1 + (\eta - 1)^{-1} $ and $X_h \in
  \mathcal{G}^2(\mathcal{T}_h)$ denotes the grid
  function generated by $(\Psi,\overline{h},\xi)$ with step sizes $h$.
\end{lemma}

\begin{proof}
  For every $1 \le i \le N$ we write the difference of the two grid
  functions as
  \begin{align*}
    e_h(t_i) &:= Z(t_i) - X_h(t_i).
  \end{align*}
  By the orthogonality of the conditional expectation it holds
  \begin{align*}
    \| e_h(t_i) \|_{L^2(\Omega;\R^d)}^2 = \big\| \E [ e_h(t_i) | \F_{t_{i-1}} ]
    \big\|^2_{L^2(\Omega;\R^d)} + \big\| e_h(t_i) - \E [ e_h(t_i) | \F_{t_{i-1}}
    ] \big\|^2_{L^2(\Omega;\R^d)}. 
  \end{align*}
  The first term is estimated as follows: Since
  \begin{align*}
    e_h(t_i) &= Z(t_i) - \Psi(Z(t_{i-1}),t_{i-1},h_{i}) +
    \Psi(Z(t_{i-1}),t_{i-1},h_{i}) - X_h(t_i)
  \end{align*}
  we first have
  \begin{align*}
    \big\| \E [ e_h(t_i) | \F_{t_{i-1}} ] \big\|_{L^2(\Omega;\R^d)}
    &\le \big\| \E \big[ Z(t_i) - \Psi(Z(t_{i-1}),t_{i-1},h_{i})
    | \F_{t_{i-1}} \big] \big\|_{L^2(\Omega;\R^d)}\\
    &\quad + \big\| \E \big[ \Psi(Z(t_{i-1}),t_{i-1},h_{i})
    - X_h(t_{i}) | \F_{t_{i-1}} \big] \big\|_{L^2(\Omega;\R^d)}.
  \end{align*}
  Then, after taking squares, it follows from the inequality $(a + b)^2 = a^2 +
  2ab +b^2 \le (1 + h_{i}^{-1}) a^2 + (1 + h_i ) b^2$ that
  \begin{align*}
    &\big\| \E [ e_h(t_i) | \F_{t_{i-1}} ] \big\|^2_{L^2(\Omega;\R^d)}\\
    &\quad \le (1 + h_{i}^{-1}) \big\| \E \big[ Z(t_i) -
    \Psi(Z(t_{i-1}),t_{i-1},h_{i}) | \F_{t_{i-1}} \big]
    \big\|_{L^2(\Omega;\R^d)}^2 \\
    &\qquad +(1 + h_i ) \big\| \E \big[ \Psi(Z(t_{i-1}),t_{i-1},h_{i})
    -  X_h(t_{i})   |\F_{t_{i-1}} \big]
    \big\|_{L^2(\Omega;\R^d)}^2.
  \end{align*}
  Replacing $h_i$ by $\eta -1$, the second term is estimated by
  \begin{align*}
    &\big\| e_h(t_i) - \E [ e_h(t_i) | \F_{t_{i-1}}
    ] \big\|^2_{L^2(\Omega;\R^d)} \\
    &\quad \le C_\eta \big\| \big( \id - \E [ \, \cdot \, | \F_{t_{i-1}}
    ] \big) \big( Z(t_i) - \Psi(Z(t_{i-1}),t_{i-1},h_{i}) \big)
    \big\|^2_{L^2(\Omega;\R^d)} \\
    &\qquad + \eta \big\| \big( \id - \E [ \, \cdot \, | \F_{t_{i-1}} ]
    \big) \big( \Psi(Z(t_{i-1}),t_{i-1},h_{i}) -
    X_h(t_{i}) \big) \big\|^2_{L^2(\Omega;\R^d)},
  \end{align*}
  where $C_\eta = 1 + (\eta - 1)^{-1}$. To sum up, we have shown that
  \begin{align*}
    &\big\| Z(t_i) - X_h(t_i) \big\|_{L^2(\Omega;\R^d)}^2\\
    &\quad \le (1 + h_{i}^{-1}) \big\| \E \big[ Z(t_i) -
    \Psi(Z(t_{i-1}),t_{i-1},h_{i}) | \F_{t_{i-1}} \big]
    \big\|_{L^2(\Omega;\R^d)}^2\\
    &\qquad + (1 + h_i ) \big\| \E \big[ \Psi(Z(t_{i-1}),t_{i-1},h_{i})
    -  X_h(t_{i}) | \F_{t_{i-1}} \big]
    \big\|_{L^2(\Omega;\R^d)}^2\\
    &\qquad + C_\eta \big\| \big( \id - \E [ \, \cdot \, | \F_{t_{i-1}} ]
    \big) \big( Z(t_i) - \Psi(Z(t_{i-1}),t_{i-1},h_{i}) \big)
    \big\|^2_{L^2(\Omega;\R^d)} \\
    &\qquad + \eta \big\| \big( \id - \E [ \, \cdot \, | \F_{t_{i-1}} ]
    \big) \big( \Psi(Z(t_{i-1}),t_{i-1},h_{i}) - X_h(t_{i}) \big) 
    \big\|^2_{L^2(\Omega;\R^d)}
  \end{align*}
  for all $1 \le i \le N$. After inserting $X_h(t_i) = \Psi(X_h(t_{i-1}),
  t_{i-1},h_i)$ and \eqref{eq:stab_cond1} we get
  \begin{align*}
    &\big\| Z(t_i) - X_h(t_i) \big\|_{L^2(\Omega;\R^d)}^2\\
    &\quad \le (1 + h_{i}^{-1}) \big\| \E \big[ Z(t_i) -
    \Psi(Z(t_{i-1}),t_{i-1},h_{i}) | \F_{t_{i-1}} \big]
    \big\|_{L^2(\Omega;\R^d)}^2\\
    &\qquad + C_\eta \big\| \big( \id - \E [ \, \cdot \, | \F_{t_{i-1}} ]
    \big) \big( Z(t_i) - \Psi(Z(t_{i-1}),t_{i-1},h_{i}) \big)
    \big\|^2_{L^2(\Omega;\R^d)} \\
    &\qquad + \big(1 + (1 + C_{\mathrm{stab}}(1 + \overline{h}))  h_i  \big)
    \big\| Z(t_{i-1}) - X_h(t_{i-1}) \big\|_{L^2(\Omega;\R^d)}^2,
  \end{align*}
  where we also made use of the fact that by \eqref{eq:stab_cond1}
  \begin{align*}
    &h_i \big\| \E \big[ \Psi(Z(t_{i-1}),t_{i-1},h_{i})
    -  X_h(t_{i}) | \F_{t_{i-1}} \big]\big\|_{L^2(\Omega;\R^d)}^2 \\
    &\quad \le h_i (1 +
    C_{\mathrm{stab}}\overline{h}) \big\| Z(t_{i-1}) - X_h(t_{i-1})
    \big\|_{L^2(\Omega;\R^d)}^2.
  \end{align*}
  Next, we subtract $\big\| Z(t_{i-1}) - X_h(t_{i-1})
  \big\|_{L^2(\Omega;\R^d)}^2$ from both sides of this inequality. Together
  with a telescopic sum argument this yields
  \begin{align*}
    &\big\| Z(t_n) - X_h(t_n) \big\|_{L^2(\Omega;\R^d)}^2 - \big\|
    Z(0) - X_h(0) \big\|_{L^2(\Omega;\R^d)}^2 \\
    &\quad = \sum_{i = 1}^n \Big( \big\| Z(t_i) - X_h(t_i)
    \big\|_{L^2(\Omega;\R^d)}^2 - \big\| Z(t_{i-1}) - X_h(t_{i-1})
    \big\|_{L^2(\Omega;\R^d)}^2 \Big)\\
    &\quad \le \sum_{i = 1}^{n} \Big( \big(1 + h_i^{-1}\big) \big\| \E \big[
    Z(t_i) - \Psi(Z(t_{i-1}), t_{i-1}, h_i) | \F_{t_{i-1}} \big]
    \big\|^2_{L^2(\Omega;\R^d)}\\ 
    &\qquad +  C_\eta \big\| \big( \id - \E \big[ \, \cdot \,
    | \F_{t_{i-1}} \big] \big) \big( Z(t_i) -
    \Psi(Z(t_{i-1}), t_{i-1}, h_i) \big) \big\|^2_{L^2(\Omega;\R^d)}\\ 
    &\qquad +(1 + C_{\mathrm{stab}}(1 + \overline{h})) h_i 
    \big\| Z(t_{i-1}) - X_h(t_{i-1}) \big\|_{L^2(\Omega;\R^d)}^2 \Big).
  \end{align*}
  After adding $\| Z(0) - X_h(0) \|_{L^2(\Omega;\R^d)}^2 = \| Z(0) - \xi
  \|_{L^2(\Omega;\R^d)}^2$ the assertion follows from an application of the
  discrete Gronwall Lemma. 
\end{proof}

A simple consequence of the stability lemma is the following estimate of the
second moment of the grid function generated by the numerical method.

\begin{corollary}
  \label{cor:bound}
  Let $(\Psi,\overline{h},\xi)$ be stochastically C-stable. If
  there exists a constant $C_0$ such that for all $(t,\delta) \in \mathbb{T}$
  it holds
  \begin{align*}
    \big\| \E \big[ \Psi(0, t, \delta) | \F_{t} \big]
    \big\|_{L^2(\Omega;\R^d)} &\le C_0 \delta,\\
    \big\| \big( \id - \E \big[ \, \cdot \,
    | \F_{t} \big] \big) \Psi(0, t, \delta)
    \big\|_{L^2(\Omega;\R^d)}&\le C_0 \delta^{\frac{1}{2}}, 
  \end{align*}
  then there exists a constant $C >0$ such that for all vectors of step sizes
  $h \in (0,\overline{h}]^N$,
  \begin{align*}
    &\max_{n \in \{0,\ldots,N\}} \| X_h(t_n)  \|_{L^2(\Omega;\R^d)}
    \le \ee^{C T}
    \Big( \| \xi  \|_{L^2(\Omega;\R^d)}^2 
    + C_0^2 (1 + \overline{h} + C_\eta) T \Big)^{\frac{1}{2}},
  \end{align*}
  where $X_h$ denotes the grid function generated by $(\Psi,\overline{h},\xi)$
  with step sizes $h$. 
\end{corollary}

\begin{proof}
  The assertion follows directly from an application of Lemma~\ref{lem:stab}
  with $Z \equiv 0 \in \mathcal{G}^2(\mathcal{T}_h)$.
\end{proof}

As the next theorem shows consistency and stability imply the strong
convergence of a stochastic one-step method. 

\begin{theorem}
  \label{th:Bconv}
  Let the stochastic one-step method $(\Psi,\overline{h},\xi)$ be
  stochastically C-stable and stochastically B-consistent of order $\gamma >
  0$. If $\xi = X_0$, then there exists a constant $C$ depending on
  $C_{\mathrm{stab}}$, $C_\mathrm{cons}$, $T$, $\overline{h}$, and $\eta$
  such that for every vector of step sizes $h \in (0,\overline{h}]^N$ it holds 
  \begin{align*}
    \max_{n \in \{0,\ldots,N\}} \big\| X(t_n) - X_h(t_n)
    \big\|_{L^2(\Omega;\R^d)} \le C |h|^{\gamma},  
  \end{align*}
  where $X$ denotes the exact solution to \eqref{sode} and $X_h$ the grid
  function generated by $(\Psi,\overline{h},\xi)$ with step sizes $h$.  In
  particular, $(\Psi,\overline{h},\xi)$ is strongly convergent of order
  $\gamma$.
\end{theorem}

\begin{proof}
  Let $h \in (0,\overline{h}]^N$ be an arbitrary vector of step sizes. Since
  $X(0) = X_h(0) = X_0$ it directly follows from Lemma~\ref{lem:stab} that
  \begin{align*}
    &\max_{n \in \{0,\ldots,N\}} \| X(t_n) - X_h(t_n)
    \|_{L^2(\Omega;\R^d)}^2\\
    &\quad \le \ee^{(1 + C_{\mathrm{stab}}(1 + \overline{h})) T} \Big( \sum_{i =
    1}^{N} \big(1 + h_i^{-1}\big) \big\| \E \big[ X(t_i) - \Psi(X(t_{i-1}),
    t_{i-1}, h_i) | \F_{t_i} \big] \big\|^2_{L^2(\Omega;\R^d)}\\ 
    &\qquad + C_\eta \sum_{i = 1}^{N} \big\| \big( \id - \E \big[ \,
    \cdot \, | \F_{t_{i-1}} \big] \big) \big( X(t_i) -
    \Psi(X(t_{i-1}), t_{i-1}, h_i) \big) \big\|^2_{L^2(\Omega;\R^d)}  \Big).
  \end{align*}
  After using \eqref{eq:cons_cond1} and \eqref{eq:cons_cond2} we get
  \begin{align*}
    &\max_{n \in \{0,\ldots,N\}} \big\| X(t_n) - X_h(t_n)
    \big\|_{L^2(\Omega;\R^d)}^2\\ 
    &\quad \le \ee^{(1 + C_{\mathrm{stab}}(1 + \overline{h})) T}
    C_{\mathrm{cons}}^2 \sum_{i = 1}^N \Big( (1 + h_i^{-1}) h_i^{2(\gamma + 1)}
    + C_\eta h_i^{2 \gamma + 1} \Big)\\ 
    &\quad \le \ee^{(1 + C_{\mathrm{stab}}(1 + \overline{h})) T}
    (1 + \overline{h} + C_\eta) T C_{\mathrm{cons}}^2 |h|^{2 \gamma}. 
   \end{align*}
   This completes the proof.
\end{proof}

\section{Solving nonlinear equations under a one-sided Lipschitz
condition}
\label{sec:nonlinear}
This section collects some results on the solvability of nonlinear equations
under a one-sided Lipschitz condition, which are needed for the error analysis
of the split-step backward Euler scheme.

The following Uniform Monotonicity Theorem is a standard
result in nonlinear analysis (see for instance, \cite[Chap.6.4]{ortega2000}, 
\cite[Theorem~C.2]{stuart1996}). We take explicit
notice of the Lip\-schitz bound for the inverse which will be used later on.
\begin{theorem}
  \label{th:monotonicity}
  Let $G \colon \R^d \to \R^d$ be a continuous mapping such that there exists a
  positive constant $c$ with
  \begin{align}
    \label{onesided}
    \langle G(x_1) - G(x_2), x_1 - x_2 \rangle \ge c | x_1 - x_2 |^2
  \end{align}
  for all $x_1, x_2 \in \R^d$. Then $G$ is a homeomorphism with Lipschitz
  continuous inverse, in particular
  \begin{align}
    \label{lipinv}
    \big| G^{-1}(y_1)-G^{-1}(y_2) \big| \le \frac{1}{c} |y_1 - y_2|
  \end{align}
  for all $y_1, y_2 \in \R^d$.
\end{theorem}

\begin{proof}
  It is well known \cite[Chap.~6.4]{ortega2000},
  \cite[Theorem~C.2]{stuart1996} that $G(x)=y$ has a unique solution for  
  every $ y\in \R^d$. Setting $x_1=G^{-1}(y_1),x_2=G^{-1}(y_2)$, condition
  \eqref{onesided} implies 
  \begin{align*}
    c |x_1-x_2 |^2 \le
    \langle y_1-y_2, x_1-x_2 \rangle 
    \le |y_1 -y_2| |x_1-x_2|,
  \end{align*}
  from which the Lipschitz estimate \eqref{lipinv} follows.
\end{proof} 

The following consequence of Theorem~\ref{th:monotonicity} contains the key
estimates for the C-stability of the split-step backward Euler scheme.
For related estimates under global Lipschitz conditions on the diffusion
coefficient functions we refer to \cite[Lemmas 3.4, 4.5]{higham2002b}.

\begin{corollary}
  \label{cor:homeomorph}
  Let the functions $f \colon [0,T] \times \R^d \to \R^d$ and $g^r \colon [0,T]
  \times \R^d \to \R^d$, $r = 1,\ldots,m$, satisfy Assumption~\ref{as:fg} with
  Lipschitz constant $L > 0$ and parameter value $\eta \in
  (\frac{1}{2},\infty)$. Let
  $\overline{h} \in (0,L^{-1})$ be given and define for every $\delta \in
  (0,\overline{h}]$ the mapping $F_\delta \colon [0,T] \times \R^d \to \R^d$ by
  $F_\delta(t,x) = x - \delta f(t, x)$. Then, the mapping $\R^d \ni x \mapsto
  F_\delta(t,x) \in \R^d$ is a homeomorphism for every $t \in [0,T]$. 
  
  In addition, the inverse $F_\delta^{-1}(t,\cdot) \colon \R^d
  \to \R^d$ satisfies
  \begin{align}
    \label{eq:Fhinv_lip}
    \big| F_\delta^{-1}(t,x_1)-F_\delta^{-1}(t,x_2) \big| &\le (1 - L
    \delta)^{-1} | x_1 - x_2 |, \\ 
    \label{eq:Fhinv_growth}
    \big| F_\delta^{-1}(t,x) \big| &\le(1 - L
    \delta)^{-1} \big( L \delta + | x | \big),
  \end{align}
  for every $x,x_1, x_2 \in \R^d$ and $t \in [0,T]$. Moreover, there
  exists a constant $C_1$ only depending on $L$ and $\overline{h}$ such
  that
  \begin{align}
    \label{eq:stab}
    \begin{split}
      &\big| F_\delta^{-1}(t,x_1) - F_\delta^{-1}(t,x_2) \big|^2 + \eta \delta 
      \sum_{r = 1}^m \big| g^r(t, F_\delta^{-1}(t,x_1)) - g^r(t,
      F_\delta^{-1}(t,x_2)) \big|^2\\
      &\qquad \le (1 + C_1 \delta) \big| x_1 - x_2 \big|^2
    \end{split}
  \end{align}
  for every $x_1, x_2 \in \R^d$ and $t \in [0,T]$.
\end{corollary}

\begin{proof}
  Fix arbitrary $\delta \in (0, \overline{h}]$ and $t \in [0,T]$.
  First, note that by \eqref{eq3:onesided} the mapping $F_\delta(t,\cdot)
  \colon \R^d \to \R^d$ is continuous and satisfies   
  \begin{align*}
    &\langle F_\delta(t,x_1) - F_\delta(t,x_2), x_1 - x_2 \rangle\\
    &\quad = | x_1 - x_2 |^2 - \delta 
    \langle f(t,x_1) - f(t,x_2), x_1 - x_2 \rangle
    \ge (1 - L \delta) | x_1 - x_2 |^2 
  \end{align*}
  for all $x_1, x_2 \in \R^d$. Note that $1- L\delta > 0$ follows from
  $\overline{h} \in (0,L^{-1})$ and $\delta \in (0,\overline{h}]$. Hence, we
  directly obtain the first assertion and \eqref{eq:Fhinv_lip} from 
  Theorem~\ref{th:monotonicity}. 
  
  Next, we set $x_0 := F_\delta(t,0) = -\delta f(t ,0) \in \R^d$. Then
  $F_\delta^{-1}(t,x_0) = 0$ and for arbitrary $x \in \R^d$ by
  \eqref{eq:Fhinv_lip} and \eqref{eq3:poly_growth} we derive
  \begin{align*}
    \big| F_\delta^{-1}(t,x) \big|
    &= \big| F_\delta^{-1}(t,x) - F_\delta^{-1}(t,x_0) \big| 
    \le(1 - L \delta)^{-1} |x - x_0|\\
    & \le (1 - L \delta)^{-1} \big( | x | + \delta |f(t,0)| \big)
    \le (1 - L \delta)^{-1} \big( | x |+ L \delta  \big). 
  \end{align*}
  It remains to give a proof of \eqref{eq:stab}. By also taking the
  diffusion coefficient functions into account, it follows from
  \eqref{eq3:onesided} that 
  \begin{align*}
    &\langle F_\delta(t,x_1) - F_\delta(t,x_2), x_1 - x_2 \rangle\\
    &\quad = | x_1 - x_2 |^2 - \delta 
    \langle f(t,x_1) - f(t,x_2), x_1 - x_2 \rangle\\
    &\quad \ge (1 - L \delta) | x_1 - x_2 |^2 + \eta \delta \sum_{r = 1}^m \big|
    g^r(t,x_1) - g^r(t,x_2) \big|^2.
  \end{align*}
  For some $y_1, y_2 \in \R^d$ we substitute $x_1 = F_\delta^{-1}(t,y_1)$ and
  $x_2 = F_\delta^{-1}(t,y_2)$ into the inequality. Then, after rearranging
  we end up with
  \begin{align*}
    &\big|F_\delta^{-1}(t,y_1) - F_\delta^{-1}(t,y_2)\big|^2 + \eta \delta
    \sum_{r = 1}^m \big| g^r(t,F_\delta^{-1}(t,y_1)) -
    g^r(t,F_\delta^{-1}(t,y_2)) \big|^2\\
    &\quad \le \big\langle y_1 - y_2, F_\delta^{-1}(t,y_1) -
    F_\delta^{-1}(t,y_2) \big\rangle + L \delta \big| F_\delta^{-1}(t,y_1) -
    F_\delta^{-1}(t,y_2)\big|^2. 
  \end{align*}
  Now, an application of \eqref{eq:Fhinv_lip}, together with the Cauchy-Schwarz
  inequality, yields
  \begin{align*}
    &\big|F_\delta^{-1}(t,y_1) - F_\delta^{-1}(t,y_2)\big|^2 +\eta \delta
    \sum_{r = 1}^m \big| g^r(t,F_\delta^{-1}(t,y_1)) -
    g^r(t,F_\delta^{-1}(t,y_2)) \big|^2\\ 
    &\quad \le |y_1 - y_2 |  \big|F_\delta^{-1}(t,y_1) -
    F_\delta^{-1}(t,y_2)\big| +  L \delta \big| F_\delta^{-1}(t,y_1) -
    F_\delta^{-1}(t,y_2)\big|^2\\
    &\quad \le(1 - L \delta)^{-1} \big( 1 + (1 - L \delta)^{-1} L \delta \big)
    |y_1 - y_2 |^2 =(1-L \delta)^{-2}|y_1-y_2|^2
  \end{align*}
  for all $y_1, y_2 \in \R^d$.
Finally, note that $b(\delta)=(1-L \delta)^{-2}$ is a convex function, hence
for all $\delta \in [0,\overline{h}]$, 
\begin{align*}
  (1 - L \delta)^{-2} \le 1 + C_1 \delta , \quad
  C_1 = \frac{b(\overline{h})-b(0)} {\overline{h}}= L(2-L\overline{h}) 
  (1- L \overline{h})^{-2},
\end{align*}
and inequality \eqref{eq:stab} is verified.
\end{proof}

The following lemma contains some further estimates of $F_{\delta}^{-1}$, which
will be useful for the analysis of the local truncation error.

\begin{lemma}
  \label{lem:est_homeo}
  Consider the same situation as in Corollary~\ref{cor:homeomorph}. 
  Then there exist constants $C_2$, $C_3$ only depending on $L$, 
  $\overline{h}$ and $q$ such that for every $\delta \in (0,\overline{h}]$ the
  inverse $F_\delta^{-1}(t,\cdot) \colon \R^d \to \R^d$ satisfies the 
  estimates 
  \begin{align}
    \label{eq:1ord}
    \big| F_\delta^{-1}(t,x) - x \big| &\le \delta C_2 \big( 1 + |
    x|^q \big),\\
    \label{eq:2ord}
    \big| F_\delta^{-1}(t,x) - x - \delta f(t,x) \big| &\le \delta^2 C_3 \big(
    1 + |x|^{2q-1} \big) 
  \end{align}
  for every $x \in \R^d$ and $t \in [0,T]$. 
\end{lemma}

\begin{proof}
  Let $x \in \R^d$ be arbitrary. For the proof of \eqref{eq:1ord} we get from
  \eqref{eq:Fhinv_lip} that
  \begin{align*}
    \big| F_\delta^{-1}(t,x) -  x \big| = \big| F_\delta^{-1}(t,x) -  
    F_\delta^{-1}(t, F_\delta(t,x)) \big| \le (1 - L
    \delta)^{-1} \delta \big| f(t,x) \big|.
  \end{align*}
  After inserting \eqref{eq3:poly_growth} and since $\delta \le \overline{h}$
  we arrive at
  \begin{align*}
    \big| F_\delta^{-1}(t,x) -  x \big| \le L (1 - L \overline{h})^{-1} \delta
    \big( 1 + |x|^q \big),
  \end{align*}
  which is \eqref{eq:1ord} with $C_2 = L (1 - L\overline{h})^{-1}$. 
  
  Next, by making use of the substitution $y = F_\delta^{-1}(t,x)$ as well as
  \eqref{eq3:loc_Lip} we obtain
  \begin{align*}
    &\big| F_\delta^{-1}(t,x) - x - \delta f(t,x) \big|
    = \big| y - F_\delta(t,y) - \delta f(t,F_\delta(t,y)) \big| \\
    &\quad = \delta \big| f(t,y) - f(t,F_\delta(t,y)) \big| \\
    &\quad \le L \delta \big( 1 + |y|^{q-1} + | F_\delta(t,y) |^{q-1} \big)
    \big| y - F_\delta(t,y) \big| \\
    &\quad \le L \delta \big( 1 + |x|^{q-1} + | F_\delta^{-1}(t,x) |^{q-1}
    \big) \big| F_\delta^{-1}(t,x) - x \big|
  \end{align*}
  for every $x \in \R^d$. After inserting \eqref{eq:1ord} and
  \eqref{eq:Fhinv_growth} we find that
  \begin{align*}
    \big| F_\delta^{-1}(t,x) - x - \delta f(t,x) \big|
    &\le C_2 L \delta^2 \big( 1 + |x|^q \big) \big( 1  + |x|^{q-1} + |
    F_\delta^{-1}(t,x) |^{q-1} \big)\\ 
    &\le \delta^2 C_3 \big( 1 + |x|^{2q-1} \big)
  \end{align*}
  for a suitable constant $C_3$ only depending on $q$, $L$, and $\overline{h}$.
\end{proof}

\section{C-stability and B-consistency of the SSBE method}
\label{sec:SSBE}
In Section~\ref{sec:def} we derived a strong convergence result in an
abstract framework. Using Section~\ref{sec:nonlinear} we are
now able to verify that the split-step backward Euler scheme from 
Example~\ref{ex:SSBE1} is stable and consistent with order $\gamma =
\frac{1}{2}$. 

Let us first show that the SSBE method is indeed a
well-defined stochastic one-step method in the sense of
Definition~\ref{def:onestep}. In Section~\ref{sec:nonlinear} we saw that the
implicit step of the SSBE method admits a unique solution if $f$ satisfies
Assumption~\ref{as:fg} with one-sided Lipschitz constant $L$. To be more
precise, let $\overline{h} \in (0,L^{-1})$ 
and consider an arbitrary vector of step sizes $h \in (0,\overline{h}]^N$.
Then, we obtain from Corollary~\ref{cor:homeomorph} that for every $1 \le i \le
N$ there exists a homeomorphism $F_{h_i}(t_i,\cdot) \colon \R^d \to \R^d$ such
that $\overline{X}_h^{\mathrm{SSBE}}(t_i) =
F_{h_i}^{-1}(t_i,X_h^{\mathrm{SSBE}}(t_{i-1}))$ is the solution to 
\begin{align*}
  \overline{X}_h^{\mathrm{SSBE}}(t_i) &= X_h^{\mathrm{SSBE}}(t_{i-1}) + h_i
  f(t_{i}, \overline{X}_h^{\mathrm{SSBE}}(t_i)).
\end{align*}
Hence, we define the one-step map $\Psi^{\mathrm{SSBE}} \colon \R^d
\times \mathbb{T} \times \Omega \to \R^d$ of the split-step backward Euler
method by
\begin{align}
  \label{eq:PsiSSBE}
  \Psi^{\mathrm{SSBE}}(x,t,\delta) = F_{\delta}^{-1}(t+\delta,x) +
  \sum_{r=1}^m g^r(t + \delta, F_\delta^{-1}(t+\delta, x)) \Delta_\delta
  W^r(t)
\end{align}
for every $x \in \R^d$ and $(t,\delta) \in \mathbb{T}$, where $\Delta_\delta
W^r(t) := W^r(t + \delta) - W^r(t)$. Next, we verify that $\Psi^{\mathrm{SSBE}}$
satisfies condition \eqref{eq:Psicond} and the assumptions of Corollary
\ref{cor:bound}.

\begin{prop}  
  \label{prop:SSBE}
  Let the functions $f$ and $g^r$, $r = 1,\ldots,m$, satisfy
  Assumption~\ref{as:fg} with $L \in (0,\infty)$ and $q \in (1,\infty)$ and let
  $\overline{h} \in (0, L^{-1})$. For every initial value $\xi \in
  L^2(\Omega;\F_{0},\P;\R^d)$ 
  it holds that $(\Psi^{\mathrm{SSBE}}, \overline{h}, \xi)$ is a stochastic
  one-step method.  
  
  In addition, there exists a constant $C_0$,   which depends on $L$, $q$, $m$,
  and $\overline{h}$, such that 
  \begin{align}
    \label{eq:SSBEzero1}
    \big\| \E \big[ \Psi^{\mathrm{SSBE}}( 0, t,\delta) | \F_{t} \big]
    \big\|_{L^2(\Omega;\R^d)} &\le C_0 \delta,\\
    \label{eq:SSBEzero2}
    \big\| \big( \id - \E [ \, \cdot\, | \F_{t} ] \big)
    \Psi^{\mathrm{SSBE}}( 0, t,\delta) \big\|_{L^2(\Omega;\R^d)} & \le C_0
    \delta^{\frac{1}{2}}    
  \end{align}
  for all $(t,\delta) \in \mathbb{T}$.
\end{prop}

\begin{proof}
  For the first assertion we only have to verify that $\Psi^{\mathrm{SSBE}}$ 
  satisfies \eqref{eq:Psicond}. For this we fix arbitrary $(t, \delta) \in
  \mathbb{T}$ and $Z \in L^2(\Omega,\F_t,\P;\R^d)$. Then, we obtain from
  Corollary~\ref{cor:homeomorph} that the mapping $F_\delta^{-1}(t+\delta,
  \cdot) \colon \R^d \to \R^d$ is a homeomorphism satisfying the linear growth
  bound \eqref{eq:Fhinv_growth}. Hence, we have 
  \begin{align*}
    F_\delta^{-1}(t+\delta, Z) \in L^2(\Omega,\F_t,\P;\R^d).
  \end{align*}
  Consequently, by the continuity of $g^r$ the mapping
  \begin{align*}
    \Omega \ni \omega \mapsto g^r(t+\delta, F_\delta^{-1}(t+\delta, Z(\omega)))
    \in \R^d
  \end{align*}
  is $\F_t / \B(\R^d)$-measurable for every $r = 1,\ldots,m$. Therefore,
  $\Psi^{\mathrm{SSBE}}(Z,t,\delta) \colon \Omega \to \R^d$ is a well-defined
  random variable, which is $\F_{t+\delta} / \B(\R^d)$-measurable. It remains
  to show that $\Psi^{\mathrm{SSBE}}(Z,t,\delta)$ is square integrable. 
  
  For this we first consider the case that $Z = 0 \in L^2(\Omega;\R^d)$. Then
  it is evident that $\Psi^{\mathrm{SSBE}}( 0, t,\delta) \in
  L^2(\Omega,\F_{t + \delta},\P;\R^d)$. In particular, it follows from
  \eqref{eq:Fhinv_growth} that 
  \begin{align*}
    \big\| \E \big[ \Psi^{\mathrm{SSBE}}( 0, t,\delta) | \F_{t} \big]
    \big\|_{L^2(\Omega;\R^d)} =
    \big| F_\delta^{-1}(t+\delta, 0) \big| \le (1 - L\delta)^{-1} L \delta 
    \le (1 - L \overline{h})^{-1} L \delta.
  \end{align*}
  Further, from an application of It\=o's isometry, \eqref{eq3:poly_growth} and
  \eqref{eq:Fhinv_growth} we get 
  \begin{align*}
    &\big\| \big( \id - \E [ \, \cdot\, | \F_{t} ] \big)
    \Psi^{\mathrm{SSBE}}( 0, t,\delta) \big\|_{L^2(\Omega;\R^d)}^2\\
    &\quad = \Big\| \sum_{r = 1}^m g^r(t+\delta, F_\delta^{-1}(t+\delta, 0) )
    \big( W^r(t+\delta) - W^r(t) \big) \Big\|^2_{L^2(\Omega;\R^d)}\\
    &\quad = \delta \sum_{r = 1}^m \big| g^r(t+\delta, F_\delta^{-1}(t+\delta,
    0) ) \big|^2\\
    &\quad \le L^2 m \delta \big( 1 + \big| F_\delta^{-1}(t+\delta,
    0) \big|^q \big)^2 \le L^2 m\big( 1 + (1 - L \overline{h})^{-q} L^q
    \overline{h}^q \big)^2 \delta.
  \end{align*}
  This verifies \eqref{eq:SSBEzero1} and \eqref{eq:SSBEzero2}. 

  Next, for arbitrary $Z \in L^2(\Omega;\F_t,\P;\R^d)$ we compute by similar
  arguments
  \begin{align*}
    &\big\| \Psi^{\mathrm{SSBE}}(Z,t,\delta) -  \Psi^{\mathrm{SSBE}}( 0,
    t,\delta) \big\|^2_{L^2(\Omega;\R^d)}\\
    &\quad = \big\|  F_\delta^{-1}(t+\delta, Z) - F_\delta^{-1}(t+\delta, 0)
    \big\|_{L^2(\Omega;\R^d)}^2 \\
    &\qquad + \delta \sum_{r = 1}^m \big\|
    g^r(t+\delta, F_\delta^{-1}(t+\delta, Z) ) -
    g^r(t+\delta, F_\delta^{-1}(t+\delta, 0) )
    \big\|_{L^2(\Omega;\R^d)}^2\\
    &\quad = \E \Big[ \big| F_\delta^{-1}(t+\delta, Z) -
    F_\delta^{-1}(t+\delta, 0) \big|^2 \\
    &\qquad + \delta \sum_{r = 1}^m \big|
    g^r(t+\delta, F_\delta^{-1}(t+\delta, Z) ) -
    g^r(t+\delta, F_\delta^{-1}(t+\delta, 0) )
    \big|^2  \Big].
  \end{align*}
  Thus, an application of \eqref{eq:stab} yields
  \begin{align*}
    &\big\| \Psi^{\mathrm{SSBE}}(Z,t,\delta) -  \Psi^{\mathrm{SSBE}}( 0,
    t,\delta) \big\|^2_{L^2(\Omega;\R^d)} \le (1 + C_1 \delta) \| Z
    \|^2_{L^2(\Omega;\R^d)}.    
  \end{align*}
  This completes the proof.
\end{proof}

\begin{theorem}
  \label{th:SSBEstab}
  Let the functions $f$ and $g^r$, $r = 1,\ldots,m$, satisfy
  Assumption~\ref{as:fg} with $L \in (0,\infty)$ and $\eta
  \in (1,\infty)$. Further, let 
  $\overline{h} \in (0, L^{-1})$. Then, for every $\xi \in
  L^2(\Omega,\F_0,\P;\R^d)$ the SSBE scheme
  $(\Psi^{\mathrm{SSBE}},\overline{h},\xi)$ is stochastically C-stable.
\end{theorem}

\begin{proof}
  Let us consider arbitrary $(t,\delta) \in \mathbb{T}$
  and $Y, Z \in L^2(\Omega,\F_t,\P;\R^d)$. For the proof of
  \eqref{eq:stab_cond1} we first note that 
  \begin{align*}
    \E \big[ \Psi^{\mathrm{SSBE}}(Y,t,\delta) -
    \Psi^{\mathrm{SSBE}}(Z,t,\delta) | \F_{t} \big]
     =  F_{\delta}^{-1}(t + \delta,Y) -
    F_{\delta}^{-1}(t + \delta,Z)
  \end{align*}
  and
  \begin{align*}
    &\big( \id - \E [ \, \cdot \, | \F_{t} ] \big)
    \big( \Psi^{\mathrm{SSBE}}(Y,t,\delta) -
    \Psi^{\mathrm{SSBE}}(Z,t,\delta)  \big)\\
    &\quad = \sum_{r = 1}^m \big(
    g^r(t + \delta,F_{\delta}^{-1}(t + \delta,Y)) - 
    g^r(t + \delta,F_{\delta}^{-1}(t + \delta,Z)) \big) \Delta_\delta W^r(t).
  \end{align*}
  Then we obtain from \eqref{eq:stab}
  \begin{align*}
      &\big\| F_{\delta}^{-1}(t + \delta,Y) -
      F_{\delta}^{-1}(t + \delta,Z) \big\|_{L^2(\Omega;\R^d)}^2\\
      &\quad + \eta \Big\| \sum_{r = 1}^m \big(
      g^r(t + \delta,F_{\delta}^{-1}(t + \delta,Y)) - 
      g^r(t + \delta,F_{\delta}^{-1}(t + \delta,Z)) \big) \Delta_\delta W^r(t) 
      \Big\|^2_{L^2(\Omega;\R^d)} \\ 
      &= \E \Big[ \big| F_{\delta}^{-1}(t + \delta,Y) -
      F_{\delta}^{-1}(t + \delta,Z) \big|^2 \\
      &\qquad + \eta \delta \sum_{r = 1}^m \big|
      g^r(t + \delta,F_{\delta}^{-1}(t + \delta,Y)) - 
      g^r(t + \delta,F_{\delta}^{-1}(t + \delta,Z)) \big|^2 \Big]\\
      &\le (1 + C_1 \delta) \big\| Y - Z \big\|^2_{L^2(\Omega;\R^d)}.
  \end{align*}
  which is condition \eqref{eq:stab_cond1} for the SSBE method with
  $C_{\mathrm{stab}} = C_1 $.
\end{proof}

The following fact is a consequence of Theorem~\ref{th:SSBEstab} and
Corollary~\ref{cor:bound} together with \eqref{eq:SSBEzero1} and
\eqref{eq:SSBEzero2}. 

\begin{corollary}
  \label{cor:SSBEbound}
  Let the functions $f$ and $g^r$, $r = 1,\ldots,m$, satisfy
  Assumption~\ref{as:fg} with $L \in (0,\infty)$ and $\eta \in (1,\infty)$. Let
  $\overline{h} \in (0, L^{-1})$. Then, for every vector of step sizes $h \in
  (0,\overline{h}]^N$ it holds for the grid function $X_h$ generated by
  $(\Psi^{\mathrm{SSBE}},\overline{h},\xi)$ that 
  \begin{align*}
    \max_{n \in \{0,\ldots,N\}} \| X_h^{\mathrm{SSBE}}(t_n)
    \|_{L^2(\Omega;\R^d)} \le \ee^{C T} \Big(
    \| \xi \|^2_{L^2(\Omega;\R^d)} + C_0^2( 1 + \overline{h} + C_\eta ) T
    \Big)^{\frac{1}{2}},
  \end{align*}
  where the constant $C_0$ is the same as in
  Proposition~\ref{prop:SSBE}.
\end{corollary}

In preparation of the proof of consistency we state the following result on the
H\"older continuity of the exact solution to \eqref{sode} with respect to the
norm in $L^p(\Omega;\R^d)$.

\begin{prop}
  \label{prop:Hoelder}
  Let $f$ and $g^r$, $r = 1,\ldots,m$, satisfy
  Assumption~\ref{as:fg} with $L \in (0,\infty)$ and $q \in (1,\infty)$.
  For every $p \in [2,\infty)$ with $\sup_{t \in [0,T]}
  \|X(t)\|_{L^{pq}(\Omega;\R^d)} < \infty$ there exists a constant $C$ such
  that 
  \begin{align*}
    \big\| X(t_1) - X(t_2) \big\|_{L^p(\Omega;\R^d)} \le C \big(1 +
    \sup_{t \in [0,T]} \|X(t)\|_{L^{pq}(\Omega;\R^d)}^q \big)
    | t_1 - t_2|^{\frac{1}{2}}
  \end{align*}
  for all $t_1, t_2 \in [0,T]$, where $X$ denotes the exact solution to
  \eqref{sode}.
\end{prop}

\begin{proof}
  Let $0 \le t_1 < t_2 \le T$. After inserting \eqref{exact} we get
  \begin{align*}
    \big\| X(t_1) - X(t_2) \big\|_{L^p(\Omega;\R^d)} &\le
    \int_{t_1}^{t_2} \big\| f(\tau,X(\tau)) \big\|_{L^p(\Omega;\R^d)}
    \diff{\tau} \\
    &\quad + \Big\| \sum_{r = 1}^m \int_{t_1}^{t_2} g^r(\tau,X(\tau))
    \diff{W^r(\tau)} \Big\|_{L^p(\Omega;\R^d)}.
  \end{align*}
  For the drift integral it follows from \eqref{eq3:poly_growth} that
  \begin{align*}
    \int_{t_1}^{t_2} \big\| f(\tau,X(\tau)) \big\|_{L^p(\Omega;\R^d)}
    \diff{\tau} \le L \big( 1 + \sup_{\tau \in [0,T]} \| X(\tau)
    \|_{L^{pq}(\Omega;\R^d)}^q \big) |t_1 - t_2|.
  \end{align*}
  In addition, the Burkholder-Davis-Gundy inequality yields
  \begin{align*}
    \Big\| \sum_{r = 1}^m \int_{t_1}^{t_2} g^r(\tau,X(\tau))
    \diff{W^r(\tau)} \Big\|_{L^p(\Omega;\R^d)} \le C \Big( \sum_{r = 1}^m
    \int_{t_1}^{t_2} \big\| g^r(\tau,X(\tau)) \big\|^2_{L^p(\Omega;\R^d)}
    \diff{\tau} \Big)^{\frac{1}{2}}
  \end{align*}
  for a constant $C = C(p)$. Then, we deduce from \eqref{eq3:poly_growth} that
  \begin{align*}
    \big\| g^r(\tau,X(\tau)) \big\|_{L^p(\Omega;\R^d)} \le
    L \big( 1  + \sup_{\tau \in [0,T]} \big\| X(\tau)
    \big\|_{L^{pq}(\Omega;\R^d)}^q \big). 
  \end{align*}
  Therefore, it holds
  \begin{align*}
    &\Big\| \sum_{r = 1}^m \int_{t_1}^{t_2} g^r(\tau,X(\tau))
    \diff{W^r(\tau)} \Big\|_{L^p(\Omega;\R^d)}\\
    &\quad \le C L m^{\frac{1}{2}} \big( 1  + \sup_{\tau \in [0,T]} \big\|
    X(\tau) \big\|_{L^{pq}(\Omega;\R^d)}^q \big) | t_1 - t_2 |^{\frac{1}{2}}.
  \end{align*}
  This completes the proof.  
\end{proof}

The following two lemmas contain estimates, which play important roles in the 
proofs of consistency for the SSBE scheme and the PEM method.

\begin{lemma}
  \label{lem:cons1} 
  Let Assumption~\ref{as:fg} be satisfied by $f$ and $g^r$, $r = 1,\ldots,m$,
  with $L \in (0,\infty)$ and $q \in (1,\infty)$.
  Further, let the exact solution $X$ to \eqref{sode} 
  satisfy $\sup_{t \in [0,T]} \| X(t) \|_{L^{4q-2}(\Omega;\R^d)} <
  \infty$. Then, there exists a constant $C$ such that for all $t_1, t_2 \in
  [0,T]$ and $s_1, s_2 \in [t_1,t_2]$ it holds
  \begin{align*}
    &\int_{t_1}^{t_2} \big\| f(\tau,X(\tau)) - f(s_1,X(s_2))
    \big\|_{L^2(\Omega;\R^d)} \diff{\tau}\\
    & \quad \le C \big( 1 + \sup_{t \in [0,T]} \big\| X(t)
    \big\|_{L^{4q-2}(\Omega;\R^d)}^{2q-1} \big) |t_1 - t_2|^{\frac{3}{2}}.
  \end{align*}
\end{lemma}

\begin{proof}
  It follows from \eqref{eq3:loc_Lip_t} and \eqref{eq3:loc_Lip} that 
  \begin{align*}
      \big| f(\tau_1,x_1) - f(\tau_2,x_2) \big| &\le \big| f(\tau_1,x_1) -
      f(\tau_1,x_2)  \big| + \big| f(\tau_1,x_2) - f(\tau_2,x_2) \big|\\
      &\le L \big( 1 + |x_1|^{q-1} + |x_2|^{q-1}  \big) |x_1 - x_2 |
      + L \big( 1 + |x_2|^q \big) |\tau_1 - \tau_2 |^{\frac{1}{2}}
  \end{align*}
  for all $\tau_1,\tau_2 \in [0,T]$ and $x_1, x_2 \in \R^d$. By
  an additional application of H\"older's inequality with exponents $\rho = 2 -
  \frac{1}{q}$ and $\rho' = \frac{2q - 1}{q-1}$ we get for all
  $s_1, s_2, \tau \in [t_1, t_2]$
  \begin{align}
    \label{eq:term6}
    \begin{split}
      &\big\|  f(\tau,X(\tau)) - f(s_1,X(s_2)) \big\|_{L^2(\Omega;\R^d)}\\
      &\quad \le L  \big\|  \big( 1 + |X(\tau)|^{q-1} +
      |X(s_2)|^{q-1}  \big) | X(\tau) - X(s_2) | \big\|_{L^2(\Omega;\R)}\\
      &\qquad + L \big\| \big( 1 + |X(s_2)|^q \big) |\tau - s_1|^{\frac{1}{2}}
      \big\|_{L^2(\Omega;\R^d)}\\
      &\quad \le L  \big( 1 + 2 \sup_{t \in [0,T]}
      \|X(t)\|^{q-1}_{L^{2\rho'(q-1)}(\Omega;\R^d)}  \big) \| X(\tau) -
      X(s_2) \|_{L^{2 \rho}(\Omega;\R^d)}\\
      &\qquad +  L \big( 1 + \sup_{t \in [0,T]}
      \|X(t)\|^{q}_{L^{2q}(\Omega;\R^d)}  \big) |t_1 -
      t_2|^{\frac{1}{2}}.
    \end{split}
  \end{align}
  Observe that $2 \rho' (q-1) = 4q -2$ and $2q \le 4q - 2$ for $q \in
  (1,\infty)$. Moreover,
  Proposition~\ref{prop:Hoelder} with $p = 2\rho$ yields 
  \begin{align*}
    \| X(\tau) - X(s_2) \|_{L^{2 \rho}(\Omega;\R^d)} &\le C \big(1 +
    \sup_{t \in [0,T]} \|X(t)\|_{L^{2 \rho q}(\Omega;\R^d)}^q \big)
    | \tau - s_2|^{\frac{1}{2}}\\
    &\le C \big(1 + \sup_{t \in [0,T]} \|X(t)\|_{L^{4 q-2}(\Omega;\R^d)}^q
    \big) |t_1 - t_2 |^{\frac{1}{2}}.
  \end{align*}
  Altogether, this proves
  \begin{align*}
    \big\|  f(\tau,X(\tau)) - f(s_1,X(s_2)) \big\|_{L^2(\Omega;\R^d)}
    \le C \big(1 + \sup_{t \in [0,T]} \|X(t)\|_{L^{4
    q-2}(\Omega;\R^d)}^{2q-1} \big) |t_1 - t_2 |^{\frac{1}{2}}
  \end{align*}
  for all $s_1,s_2,\tau \in [t_1, t_2]$. After integrating over $\tau \in
  [t_1,t_2]$ the proof is completed. 
\end{proof}

\begin{lemma}
  \label{lem:cons2} 
  Let Assumption~\ref{as:fg} be satisfied by $f$ and $g^r$, $r = 1,\ldots,m$.
  Further, let the exact solution $X$ to \eqref{sode} 
  satisfy $\sup_{\tau \in [0,T]} \| X(\tau) \|_{L^{4q-2}(\Omega;\R^d)} <
  \infty$. Then, there exists a constant $C$ such that for all $t_1, t_2,s \in
  [0,T]$ with $0 \le t_1 \le s \le t_2 \le T$ it holds
  \begin{align*}
    &\Big\| \sum_{r = 1}^m \int_{t_1}^{t_2} \big(
    g^r(\tau,X(\tau)) - g^r(s,X(t_1)) \big) \diff{W^r(\tau)}
    \Big\|_{L^2(\Omega;\R^d)}\\
    & \quad \le C m^{\frac{1}{2}} \big( 1 + \sup_{\tau \in [0,T]} \big\| X(\tau)
    \big\|_{L^{4q-2}(\Omega;\R^d)}^{2q-1} \big) |t_1 - t_2|.
  \end{align*}
\end{lemma}

\begin{proof}
  By the It\=o isometry we get
  \begin{align*}
    &\Big\| \sum_{r = 1}^m \int_{t_1}^{t_2} \big(
    g^r(\tau,X(\tau)) - g^r(s,X(t_1)) \big) \diff{W^r(\tau)}
    \Big\|_{L^2(\Omega;\R^d)}\\
    &\quad = \Big( \sum_{r = 1}^m \int_{t_1}^{t_2} \big\|
    g^r(\tau,X(\tau)) - g^r(s,X(t_1)) \big\|^2_{L^2(\Omega;\R^d)} \diff{\tau}
    \Big)^{\frac{1}{2}}.
  \end{align*}
  Then, the integrands are estimated in the same way as in \eqref{eq:term6} by
  \begin{align*}
    &\big\| g^r(\tau,X(\tau)) - g^r(s,X(t_1)) \big\|_{L^2(\Omega;\R^d)} \\
    &\quad \le L  \big( 1 + 2 \sup_{t \in [0,T]}
    \|X(t)\|^{q-1}_{L^{2 \rho'(q-1)}(\Omega;\R^d)}  \big) \| X(\tau) - X(t_1)
    \|_{L^{2\rho}(\Omega;\R^d)}\\
    &\qquad + L \big( 1 + \sup_{t \in [0,T]}
    \|X(t)\|^{q}_{L^{2q}(\Omega;\R^d)}  \big) |\tau - s |^{\frac{1}{2}}\\
    &\quad \le C \big( 1 + \sup_{t \in [0,T]} \|X(t)\|^{2q-1}_{L^{4
    q-2}(\Omega;\R^d)} \big) |t_1 - t_2|^{\frac{1}{2}},
  \end{align*}
  where we again made use of the $\frac{1}{2}$-H\"older continuity of the exact
  solution.  
\end{proof}

The next theorem finally investigates the B-consistency of the SSBE method.

\begin{theorem}
  \label{th:SSBEcons}
  Let the functions $f$ and $g^r$, $r = 1,\ldots,m$, satisfy
  Assumption~\ref{as:fg} with $L \in (0,\infty)$ and $q \in (1,\infty)$. Let
  $\overline{h} \in (0,L^{-1})$. If the exact solution $X$ to \eqref{sode}
  satisfies $\sup_{\tau \in [0,T]} \| X(\tau) \|_{L^{4q-2}(\Omega;\R^d)} <
  \infty$, then the split-step backward Euler method
  $(\Psi^{\mathrm{SSBE}},\overline{h},X_0)$ is stochastically B-consistent of
  order $\gamma = \frac{1}{2}$.  
\end{theorem}

\begin{proof}
  Let $(t,\delta) \in \mathbb{T}$ be arbitrary. First we insert \eqref{exact}
  and \eqref{eq:PsiSSBE} and obtain
  \begin{align*}
    &X(t + \delta) - \Psi^{\mathrm{SSBE}} (X(t),t,\delta)
    = \int_{t}^{t + \delta} \big( f(\tau,X(\tau)) - f(t + \delta,X(t))
    \big) \diff{\tau}\\
    &\qquad + X(t) + \delta f(t + \delta,X(t)) -
    F_{\delta}^{-1}(t+\delta,X(t))\\
    &\qquad + \sum_{r = 1}^m \int_{t}^{t + \delta} \big( g^r(\tau,X(\tau)) -
    g^r(t + \delta,X(t)) \big) \diff{W^r(\tau)}\\
    &\qquad +\sum_{r = 1}^m \big( g^r(t+ \delta,X(t)) -
    g^r(t+\delta,F_{\delta}^{-1}(t + \delta,X(t))) \big) \Delta_\delta W^r(t).
  \end{align*}
  For the proof of \eqref{eq:cons_cond1} we therefore have to estimate
  \begin{align}
    \label{eq:term4}
    \begin{split}
      &\big\| \E \big[ X(t + \delta) -\Psi^{\mathrm{SSBE}} (X(t),t,\delta) |
      \F_{t} \big] \big\|_{L^2(\Omega;\R^d)}\\
      &\quad \le \int_{t}^{t + \delta} \big\| \E \big[ f(\tau,X(\tau)) -
      f(t + \delta,X(t)) | \F_{t} \big] \big\|_{L^2(\Omega;\R^d)}
      \diff{\tau}\\
      &\qquad + \big\| X(t) + \delta f(t+\delta,X(t)) -
      F_{\delta}^{-1}(t + \delta,X(t)) \big\|_{L^2(\Omega;\R^d)}.
    \end{split}
  \end{align}
  Together with the inequality $\| \E[ Y | \F_t ] \|_{L^2(\Omega;\R^d)} \le \| Y
  \|_{L^2(\Omega;\R^d)}$ for all $Y \in L^2(\Omega;\R^d)$ it follows 
  from Lemma~\ref{lem:cons1} that   
  \begin{align*}
    \int_{t}^{t + \delta} \big\| \E \big[ f(\tau,X(\tau)) - f(t + \delta,X(t))
    | \F_{t} \big] \big\|_{L^2(\Omega;\R^d)} \diff{\tau}
    \le C_{\mathrm{cons}} \delta^{\frac{3}{2}}
  \end{align*}
  for a constant $C_{\mathrm{cons}}$ depending on $L$, $q$, $m$, and
  $\sup_{\tau \in [0,T]} \|X(\tau)\|^{2q-1}_{L^{4q-2}(\Omega;\R^d)}$. In order
  to complete the 
  proof 
  of \eqref{eq:cons_cond1} we need to show a similar estimate of the second
  term in \eqref{eq:term4}. In fact, it follows from \eqref{eq:2ord} that
  \begin{align*}
    &\big\| X(t) + \delta f(t + \delta,X(t)) -
    F_{\delta}^{-1}(t + \delta,X(t)) \big\|_{L^2(\Omega;\R^d)}\\
    &\quad \le C_3 \delta^{2} \big\| 1 + | X(t) |^{2q-1}
    \big\|_{L^2(\Omega;\R)} 
    \le C_3 \delta^2 \Big(1 + \sup_{\tau \in [0,T]} \| X(\tau)
    \|_{L^{4q - 2}(\Omega;\R^d)}^{2q-1} \Big).
  \end{align*}
  This completes the proof of \eqref{eq:cons_cond1} with $\gamma =
  \frac{1}{2}$ and we turn our attention to the proof of \eqref{eq:cons_cond2}.
  For this we need to estimate the following three terms
  \begin{align}
    \label{eq:term5}
    \begin{split}
      &\big\| \big( \id - \E [ \, \cdot \, | \F_{t} ] \big)
      \big( X(t + \delta) - \Psi^{\mathrm{SSBE}}(X(t),t,\delta) \big)
      \big\|_{L^2(\Omega;\R^d)} \\
      &\quad \le \int_{t}^{t + \delta} \big\| \big( \id - \E [ \, \cdot
      \, | \F_{t} ] \big) \big( f(\tau,X(\tau))  - f(t+\delta, X(t)) \big) 
      \big\|_{L^2(\Omega;\R^d)} \diff{\tau}\\
      &\qquad + \Big\| \sum_{r = 1}^m \int_{t}^{t + \delta} \big(
      g^r(\tau,X(\tau)) - g^r(t + \delta,X(t)) \big) \diff{W^r(\tau)}
      \Big\|_{L^2(\Omega;\R^d)}\\
      &\qquad + \Big\| \sum_{r = 1}^m \big( g^r(t + \delta,X(t)) -
      g^r(t+\delta,F_{\delta}^{-1}(t + \delta,X(t))) \big) \Delta_\delta
      W^r(t) \Big\|_{L^2(\Omega;\R^d)}.
    \end{split}
  \end{align}
  For the first term we get from Lemma~\ref{lem:cons1} and since 
  $\| ( \id - \E[\,\cdot\,  | \F_t ] \big) Y \|_{L^2(\Omega;\R^d)} \le
  \| Y \|_{L^2(\Omega;\R^d)}$ for all $Y \in L^2(\Omega;\R^d)$ that
  \begin{align*}
    &\int_{t}^{t + \delta} \big\| \big( \id - \E [ \, \cdot
    \, | \F_{t} ] \big) \big( f(\tau,X(\tau)) - f(t+\delta,X(t)) \big)
    \big\|_{L^2(\Omega;\R^d)} \diff{\tau}\le C_{\mathrm{cons}}
    \delta^{\frac{3}{2}}.
  \end{align*}
  We apply Lemma~\ref{lem:cons2} to the second term in \eqref{eq:term5}. This
  yields 
  \begin{align*}
    &\Big\| \sum_{r = 1}^m \int_{t}^{t + \delta} \big(
    g^r(\tau,X(\tau)) - g^r(t + \delta,X(t)) \big) \diff{W^r(\tau)}
    \Big\|_{L^2(\Omega;\R^d)} \le C_{\mathrm{cons}} \delta.
  \end{align*}
  Finally, for the last term in \eqref{eq:term5} it follows from
  \eqref{eq3:loc_Lip}, \eqref{eq:Fhinv_growth}, and \eqref{eq:1ord} that
  \begin{align*}
    &\big| g^r(t + \delta,X(t)) - g^r(t + \delta,F_{\delta}^{-1}(t +
    \delta,X(t))) \big|\\
    &\quad \le L \big( 1 + | X(t) |^{q-1}  + |F_{\delta}^{-1}(t +
    \delta,X(t))|^{q-1} \big)
    \big|X(t) - F_{\delta}^{-1}(t + \delta,X(t)) \big| \\
    &\quad \le \delta C_2 L \big( 1 + | X(t) |^{q-1}  + (1 - L \delta)^{-(q-1)} 
    ( L \delta + | X(t) | )^{q-1} \big)\big( 1 + | X(t) |^{q} \big)\\
    &\quad \le C \delta \big( 1 + | X(t) |^{2q-1} \big)
  \end{align*}
  for a suitable constant $C$ only depending on $C_2$, $L$, $q$, and
  $\overline{h}$. Therefore,
  \begin{align*}
    &\Big\| \sum_{r = 1}^m \big( g^r(t + \delta,X(t)) -
    g^r(t + \delta,F_{\delta}^{-1}(t + \delta,X(t))) \big) \Delta_\delta
    W^r(t) \Big\|_{L^2(\Omega;\R^d)}^2\\
    &\quad = \delta \sum_{r = 1}^m \big\| g^r(t + \delta,X(t)) -
    g^r(t + \delta,F_{\delta}^{-1}(t + \delta,X(t)))
    \big\|_{L^2(\Omega;\R^d)}^2\\
    &\quad \le C^2 m \delta^3 \Big( 1 + \sup_{\tau \in [0,T]} \|
    X(\tau) \|^{2q-1}_{L^{4q-2}(\Omega;\R^d)} \Big)^2. 
  \end{align*}
  Altogether, this completes the proof of \eqref{eq:cons_cond2}.
\end{proof}

The strong convergence of the SSBE scheme follows now directly from
Theorems~\ref{th:SSBEstab} and \ref{th:SSBEcons} as well as
Theorem~\ref{th:Bconv}. 

\begin{theorem}
  \label{cor:SSBEconv}
  Let the functions $f$ and $g^r$, $r = 1,\ldots,m$, satisfy
  Assumption~\ref{as:fg} with constants $L \in (0,
  \infty)$, $\eta \in (1,\infty)$,
  and $q \in (1,\infty)$. Let $\overline{h} \in
  (0,L^{-1})$. If the exact solution $X$ to \eqref{sode}
  satisfies $\sup_{\tau \in [0,T]} \| X(\tau) \|_{L^{4q-2}(\Omega;\R^d)} <
  \infty$, then the split-step backward Euler method
  $(\Psi^{\mathrm{SSBE}},\overline{h},X_0)$ is strongly convergent of order
  $\gamma = \frac{1}{2}$. 
\end{theorem}

%
  
\begin{remark}
  \label{rem:BEM}
  Instead of the SSBE method many authors study the 
  \emph{implicit Euler-Maruyama method} or \emph{backward Euler-Maruyama
  method} (BEM) from \cite[Chap.~12]{kloeden1999}. For instance, in
  \cite{andersson2015, higham2002b, mao2013a} this scheme is considered for the
  approximation  of stochastic differential equations with super-linearly
  growing coefficient functions.
  
  Let $h = (h_1,\ldots,h_N)$ be a suitable vector of step sizes. Then, the BEM
  method is implicitly given by the recursion 
  \begin{align*}
    X_h^{\mathrm{BEM}}(t_i) &= X_h^{\mathrm{BEM}}(t_{i-1}) + h_i f(t_{i},
    X_h^{\mathrm{BEM}}(t_i)) \\
    &\quad + \sum_{r = 1}^m g^r(t_{i-1}, X_h^{\mathrm{BEM}}(t_{i-1}) ) \big(
    W^r(t_{i}) - W^r(t_{i-1}) \big),\quad 1 \le i \le N,\\
    X_h^{\mathrm{BEM}}(0) &= X_0.
  \end{align*}
  For the remainder of this remark, we assume that $h$ is a vector of
  equidistant step sizes, that is $h_i = h_j$, for all $i, j = 1,\ldots,N$.
  Further, we consider the situation of autonomous coefficient functions
  $f(t,x) = f(x)$ and $g^r(t,x) = g^r(x)$, $r = 1,\ldots,m$, for all $x \in
  \R^d$ and $t \in [0,T]$. 

  Under these additional conditions we are able to mimic an idea of proof from
  \cite[Lemma~5.1]{higham2002b}. The starting point is the observation that the
  defining recursion of the BEM method can be rewritten artificially as a
  split-step method by 
  \begin{align}
    \label{eq:BEMsplit}
    \begin{split}
      \overline{X}_h^{\mathrm{BEM}}(t_i) &= X_h^{\mathrm{BEM}}(t_{i-1})
      + \sum_{r = 1}^m g^r(X_h^{\mathrm{BEM}}(t_{i-1}) )
      \Delta_{h_i}W^r(t_{i-1}), \\
      X_h^{\mathrm{BEM}}(t_i) &= \overline{X}_h^{\mathrm{BEM}}(t_i) + h_i f(
      X_h^{\mathrm{BEM}}(t_i) )
    \end{split}
  \end{align}
  for every $1 \le i \le N$. Thus, the SSBE scheme and the BEM method only
  differ in the order, in which the implicit step for the drift part and the
  explicit step for the diffusion part are applied. Consequently, one easily
  verifies that $\overline{X}_h^{\mathrm{BEM}}$ is the grid function generated
  by the SSBE scheme $(\Psi^{\mathrm{SSBE}}, \overline{h}, \xi)$ with initial
  condition $\xi = F_{h_1}(X_0)$. Then, one can interpret the BEM method as a
  perturbation of the SSBE scheme in the following sense: By the
  homeomorphism $F_{h_i}( \cdot )$ it holds
  \begin{align}
    \label{eq:relBEM}
    X_h^{\mathrm{BEM}}(t_i) = F_{h_i}^{-1}(
    \overline{X}_h^{\mathrm{BEM}}(t_i) ).
  \end{align}
  Therefore, a strong error result for the BEM method can be derived by an
  application of the stability Lemma~\ref{lem:stab}, where $X_h^{\mathrm{BEM}}$
  takes over the role of the exact solution. To be more precise, we decompose
  the strong error of the BEM method into the following three parts
  \begin{align}
    \label{eq:errBEM}
    \begin{split}
      &\big\| X(t_n) - X_h^{\mathrm{BEM}}(t_n) \big\|_{L^2(\Omega;\R^d)} \le
      \big\| X(t_n) - X_h^{\mathrm{SSBE}}(t_n)
      \big\|_{L^2(\Omega;\R^d)}\\
      &\quad +  \big\| X_h^{\mathrm{SSBE}}(t_n) -
      \overline{X}^{\mathrm{BEM}}_h(t_n) \big\|_{L^2(\Omega;\R^d)} + \big\|
      \overline{X}_h^{\mathrm{BEM}}(t_n) - X_h^{\mathrm{BEM}}(t_n)
      \big\|_{L^2(\Omega;\R^d)}
    \end{split}
  \end{align}
  for every $n \in \{1,\ldots,N\}$. Then the first term is the strong error of
  the SSBE scheme while the second can be estimated by Lemma~\ref{lem:stab} 
  and \eqref{eq3:poly_growth}. Similarly, we derive a suitable bound for the
  third term by inserting \eqref{eq:BEMsplit} and making again use of
  \eqref{eq3:poly_growth}.
  However, this line of arguments has the disadvantage that we are in need of
  higher moment bounds for the grid function $X_h^{\mathrm{BEM}}$, uniformly
  with respect to the step size $h$. We have not been able to prove if the BEM
  method is a stochastically C-stable numerical one-step scheme under
  Assumption \ref{as:fg}. We refer to \cite{andersson2015} for a more
  direct proof of the mean-square convergence of the backward Euler method,
  which does not rely on higher moment bounds of the numerical scheme.
\end{remark}

\section{C-stability and B-consistency of the PEM method}
\label{sec:PEM}
In this section we prove that the projected Euler-Maruyama method from
Example~\ref{ex:PEM} is stochastically C-stable and B-consistent of order
$\gamma = \frac{1}{2}$. 

We begin by showing that the PEM method is a stochastic
one-step method in the sense of Definition~\ref{def:onestep}. Let
Assumption~\ref{as:fg} be satisfied with growth rate $q \in (1,\infty)$. Then
we set $\alpha = \frac{1}{2(q-1)}$ and for an arbitrary upper step size bound
$\overline{h} \in (0, 1]$ we  define the one-step map $\Psi^{\mathrm{PEM}}
\colon \R^d \times \mathbb{T} \times \Omega \to \R^d$ by 
\begin{align}
  \label{eq:PsiPEM}
  \begin{split}
    \Psi^{\mathrm{PEM}}(x,t,\delta) &:= \min(1,\delta^{-\alpha}|x|^{-1}) x 
    + \delta f(t, \min(1,\delta^{-\alpha}|x|^{-1}) x )\\
    &\quad + \sum_{r = 1}^m g^r(t, \min(1,\delta^{-\alpha}|x|^{-1}) x)
    \Delta_\delta W^r(t)    
  \end{split}
\end{align}
for every $x \in \R^d$ and $(t,\delta) \in \mathbb{T}$. As before we write
$\Delta_\delta W^r(t) = W^r(t + \delta) - W^r(t)$.

\begin{prop}  
  \label{prop:PEM}
  Let the functions $f$ and $g^r$, $r = 1,\ldots,m$, satisfy
  Assumption~\ref{as:fg} with $L \in (0,\infty)$, $q \in (1,\infty)$, and let
  $\overline{h} \in (0,1]$. For every initial value $\xi \in
  L^2(\Omega;\F_{0},\P;\R^d)$ it holds that $(\Psi^{\mathrm{PEM}},
  \overline{h}, \xi)$ with $\alpha = \frac{1}{2(q-1)}$ is a stochastic one-step
  method. 
  
  In addition, there exists a constant $C_0$ only depending on $L$ and
  $m$ such that 
  \begin{align}
    \label{eq:PEMzero1}
    \big\| \E \big[ \Psi^{\mathrm{PEM}}( 0, t,\delta) | \F_{t} \big]
    \big\|_{L^2(\Omega;\R^d)} &\le C_0 \delta,\\
    \label{eq:PEMzero2}
    \big\| \big( \id - \E [ \, \cdot\, | \F_{t} ] \big)
    \Psi^{\mathrm{PEM}}( 0, t,\delta) \big\|_{L^2(\Omega;\R^d)} & \le C_0
    \delta^{\frac{1}{2}}    
  \end{align}
  for all $(t,\delta) \in \mathbb{T}$.
\end{prop}

\begin{proof}
  As in the proof of Proposition~\ref{prop:SSBE} we first
  verify that $\Psi^{\mathrm{PEM}}$ satisfies \eqref{eq:Psicond}. Let us 
  fix arbitrary $(t, \delta) \in \mathbb{T}$ and $Z \in
  L^2(\Omega,\F_t,\P;\R^d)$.  
  
  By the continuity and boundedness of the mapping $\R^d \ni x \mapsto \min(1,
  \delta^{-\alpha} |x|^{-1} ) x$ we obtain
  \begin{align*}
    \min(1, \delta^{-\alpha} |Z|^{-1} ) Z \in L^\infty(\Omega,\F_t,\P;\R^d).
  \end{align*}
  Consequently, by \eqref{eq3:poly_growth} it also holds true that
  \begin{align*}
    f(t,\min(1, \delta^{-\alpha} |Z|^{-1} ) Z) \in L^\infty(\Omega,\F_t,\P;\R^d)
  \end{align*}
  as well as
  \begin{align*}
    g^r(t,\min(1, \delta^{-\alpha} |Z|^{-1} ) Z) \in
    L^\infty(\Omega,\F_t,\P;\R^d) 
  \end{align*}
  for every $r = 1,\ldots,m$. Therefore, $\Psi^{\mathrm{PEM}}(Z,t,\delta)
  \colon \Omega \to \R^d$ is an $\F_{t+\delta} / \B(\R^d)$-measurable random
  variable, which satisfies condition \eqref{eq:Psicond}.

  It remains to show \eqref{eq:PEMzero1} and \eqref{eq:PEMzero2}. From
  \eqref{eq3:poly_growth} it follows at once that
  \begin{align*}
    \big\| \E \big[ \Psi^{\mathrm{PEM}}( 0, t,\delta) | \F_{t} \big]
    \big\|_{L^2(\Omega;\R^d)} =
    \big| \delta f(t, 0) \big| \le  L \delta.
  \end{align*}
  Similarly, from It\=o's isometry and \eqref{eq3:poly_growth} we get
  \begin{align*}
    &\big\| \big( \id - \E [ \, \cdot\, | \F_{t} ] \big)
    \Psi^{\mathrm{PEM}}( 0, t,\delta) \big\|_{L^2(\Omega;\R^d)}^2\\
    &\quad = \Big\| \sum_{r = 1}^m g^r(t,  0 )
    \big( W^r(t+\delta) - W^r(t) \big) \Big\|^2_{L^2(\Omega;\R^d)}
    = \delta \sum_{r = 1}^m \big| g^r(t, 0 ) \big|^2
    \le L^2 m \delta.
  \end{align*}
  This verifies \eqref{eq:PEMzero1} and \eqref{eq:PEMzero2}. 
\end{proof}

For the formulation of the following lemmas we introduce the abbreviation
\begin{align}
  \label{eq:circ}
  x^\circ := \min(1, \delta^{-\alpha} |x|^{-1}) x
\end{align}
for every $x \in \R^d$ and every step size $\delta \in (0,1]$. 

\begin{lemma}
  \label{lem:circ}
  For every $\alpha \in (0,\infty)$ and $\delta \in (0,1]$ the
  mapping $\R^d \ni x \mapsto x^\circ \in \R^d$ is globally Lipschitz
  continuous with Lipschitz constant $1$. In particular, it holds
  \begin{align}
    \label{eq:PEMLip}
    \big| x^\circ_1 - x^\circ_2 \big| \le \big| x_1- x_2 \big|
  \end{align}
  for all $x_1, x_2 \in \R^d$.
\end{lemma}

\begin{proof}
  For a proof of the Lipschitz continuity we first compute 
  \begin{align*}
    \big| x^\circ_1 - x^\circ_2 \big|^2 &= | x_1 - x_2 |^2 + \big[
    \big|x^\circ_1 \big|^2 - |x_1|^2 - 2 \big( \langle
    x_1^\circ, x_2^\circ \rangle - \langle x_1, x_2 \rangle \big) 
    + \big|x^\circ_2\big|^2 - |x_2|^2 \big]
  \end{align*}
  for all $x_1, x_2 \in \R^d$. We show that the second term is always
  nonpositive.    

  This is clearly true for the case $|x_1| \le \delta^{-\alpha}$ and $|x_2| \le
  \delta^{-\alpha}$, since then $x_i = x_i^\circ$, $i \in \{1,2\}$. Therefore,
  for the rest of this proof we assume without 
  loss of generality that $|x_1| > \delta^{-\alpha}$. After inserting this into
  the second term we
  obtain from an application of the Cauchy-Schwarz inequality 
  \begin{align*}
    &\big|x^\circ_1 \big|^2 - |x_1|^2 - 2 \big( \langle
    x_1^\circ, x_2^\circ \rangle - \langle x_1, x_2 \rangle \big) 
    + \big|x^\circ_2\big|^2 - |x_2|^2\\
    &\quad = \delta^{-2\alpha} - |x_1|^2  + \min(|x_2|,
    \delta^{-\alpha} )^2 - |x_2|^2 \\
    &\qquad + 2 \big( 1 - \delta^{-\alpha} |x_1|^{-1} 
    \min(1, \delta^{-\alpha} |x_2|^{-1}) \big) \langle x_1, x_2 \rangle\\
    &\quad \le \delta^{-2\alpha} - |x_1|^2  + \min(|x_2|,
    \delta^{-\alpha} )^2 - |x_2|^2 \\
    &\qquad + 2 \big( |x_1| |x_2| - \delta^{-\alpha} \min(|x_2|,
    \delta^{-\alpha}) \big)\\
    &\quad = \big( \delta^{-\alpha} -\min(|x_2|,
    \delta^{-\alpha}) \big)^2 - \big( |x_1| - |x_2| \big)^2 \le 0,   
  \end{align*}
  since we assumed $|x_1| > \delta^{-\alpha}$. This proves the asserted
  Lipschitz continuity.
\end{proof}

The following inequality \eqref{eq:PEMcstab1} will play the same role for the
stability analysis of the PEM method as \eqref{eq:stab} does for the SSBE
scheme.

\begin{lemma}
  \label{lem:PEM}
  Let $f$ and $g^r$, $r
  = 1,\ldots,m$, satisfy Assumption~\ref{as:fg} with $L \in
  (0,\infty)$, $q \in (1,\infty)$, and $\eta \in (\frac{1}{2},\infty)$. 
  Consider the mapping $\R^d \ni x \mapsto x^\circ \in \R^d$ defined in
  \eqref{eq:circ} with $\alpha \in (0,\frac{1}{2(q-1)}]$ and $\delta \in (0,1]$.
  Then, there exists a constant $C$ only depending on $L$ with
  \begin{align}
    \label{eq:PEMcstab1}
    \begin{split}
      & \big| x^\circ_1 - x_2^\circ + \delta ( f(t,x_1^\circ) - f(t,x_2^\circ))
      \big|^2 + 2 \eta \delta \sum_{r = 1}^m \big| g^r(t,x_1^\circ) -
      g^r(t,x_2^\circ) \big|^2\\
      &\qquad \le (1 + C \delta) | x_1 - x_2 |^2
    \end{split}
  \end{align}
  for all $x_1, x_2 \in \R^d$.
\end{lemma}

\begin{proof}
  For the proof of \eqref{eq:PEMcstab1} we obtain from \eqref{eq3:onesided} 
  \begin{align*}
    &\big| x^\circ_1 - x_2^\circ + \delta ( f(t,x_1^\circ) - f(t,x_2^\circ))
    \big|^2\\
    &\quad =  \big| x^\circ_1 - x_2^\circ \big|^2 + 2 \delta \big\langle
    x^\circ_1 - x_2^\circ, f(t,x_1^\circ) - f(t,x_2^\circ) \big\rangle + 
    \delta^2 \big| f(t,x_1^\circ) - f(t,x_2^\circ) \big|^2  \\
    &\quad \le (1 + 2 L \delta) \big| x^\circ_1 - x_2^\circ \big|^2 
    - 2 \eta \delta \sum_{r = 1}^m \big| g^r(t,x_1^\circ) - g^r(t,x_2^\circ)
    \big|^2 
    + \delta^2 \big| f(t,x_1^\circ) - f(t,x_2^\circ) \big|^2
  \end{align*}
  for all $x_1, x_2 \in \R^d$. Next, applications of \eqref{eq3:loc_Lip} and
  \eqref{eq:PEMLip} yield
  \begin{align*}
     \big| f(t,x_1^\circ) - f(t,x_2^\circ) \big| &\le L \big( 1 +
     |x_1^\circ|^{q-1} + |x_2^\circ|^{q-1}  \big) \big|x_1^\circ - x_2^\circ
     \big| \\
     &\le L \big( 1 + 2 \delta^{-\alpha(q-1)} \big) \big|x_1 - x_2 \big|,
  \end{align*}
  where we also made use of the fact that $|x_1^\circ|, |x_2^\circ| \le
  \delta^{-\alpha}$. 
  Since $\alpha \in (0, \frac{1}{2(q-1)}]$ and $\delta \in (0,1]$ it follows
  that $\delta^{-\alpha(q-1)} \le \delta^{-\frac{1}{2}}$. Hence, we get
  \begin{align*}
    &\big| x^\circ_1 - x_2^\circ + \delta ( f(t,x_1^\circ) - f(t,x_2^\circ))
    \big|^2  + 2 \eta \delta \sum_{r = 1}^m \big|
    g^r(t,x_1^\circ) - g^r(t,x_2^\circ) \big|^2\\
    &\quad \le (1 + 2 L \delta) \big| x_1 - x_2 \big|^2
    + \delta^2 L^2 \big( 1 + 2 \delta^{-\frac{1}{2}} \big)^2 \big|x_1 - x_2
    \big|^2 \le (1 + C \delta) \big|x_1 - x_2 \big|^2
  \end{align*}
  with $C = 2L + 9 L^2$. 
\end{proof}

The next theorem verifies that the PEM method is stochastically C-stable.

\begin{theorem}
  \label{th:PEMstab}
  Let the functions $f$ and $g^r$, $r = 1,\ldots,m$, satisfy
  Assumption~\ref{as:fg} with $L \in (0,\infty)$, $q \in (1,\infty)$, and
  $\eta \in (\frac{1}{2},\infty)$. Further, let 
  $\overline{h} \in (0, 1]$. Then, for every $\xi \in
  L^2(\Omega,\F_0,\P;\R^d)$ the projected Euler-Maruyama method
  $(\Psi^{\mathrm{PEM}},\overline{h},\xi)$ with $\alpha = \frac{1}{2(q-1)}$ is
  stochastically C-stable.
\end{theorem}

\begin{proof}
  Let $(t,\delta) \in \mathbb{T}$ be arbitrary and consider $Y, Z \in
  L^2(\Omega,\F_t,\P;\R^d)$. By recalling the notation \eqref{eq:circ}
  we get that
  \begin{align*}
    \E \big[ \Psi^{\mathrm{PEM}}(Y,t,\delta) -
    \Psi^{\mathrm{PEM}}(Z,t,\delta) | \F_{t} \big]
     = Y^\circ + \delta f(t, Y^\circ) - ( Z^\circ + \delta f(t,Z^\circ))
  \end{align*}
  and
  \begin{align*}
    &\big( \id - \E [ \, \cdot \, | \F_{t} ] \big)
    \big( \Psi^{\mathrm{PEM}}(Y,t,\delta) -
    \Psi^{\mathrm{PEM}}(Z,t,\delta)  \big)\\
    &\quad = \sum_{r = 1}^m \big(
    g^r(t , Y^\circ) - 
    g^r(t , Z^\circ) \big) \Delta_\delta W^r(t).
  \end{align*}
  Then, from the It\=o isometry and \eqref{eq:PEMcstab1} it follows 
  \begin{align*}
      &\big\| Y^\circ + \delta f(t, Y^\circ) - ( Z^\circ + \delta f(t,Z^\circ))
      \big\|_{L^2(\Omega;\R^d)}^2\\
      &\quad + 2 \eta \Big\| \sum_{r = 1}^m \big( g^r(t , Y^\circ) - 
      g^r(t , Z^\circ) \big) \Delta_\delta W^r(t) \Big\|^2_{L^2(\Omega;\R^d)} \\
      &= \E \Big[ \big| Y^\circ + \delta f(t, Y^\circ) - ( Z^\circ + \delta
      f(t,Z^\circ)) \big|^2 +
      2 \eta \delta \sum_{r = 1}^m \big|  g^r(t , Y^\circ) - 
      g^r(t , Z^\circ) \big|^2 \Big]\\
      &\le (1 + C \delta) \big\| Y - Z \big\|^2_{L^2(\Omega;\R^d)}.
  \end{align*}
  which is condition \eqref{eq:stab_cond1} for the PEM method with
  $C_{\mathrm{stab}} =C$.
\end{proof}

It remains to show that the PEM method is stochastically $B$-consistent of
order $\gamma = \frac{1}{2}$. An important ingredient of our proof is
contained in the following lemma, which is based on an argument already
found in the proof of \cite[Theorem~2.2]{higham2002b}.

\begin{lemma}
  \label{lem:PEMcons1}
  Let $L \in (0,\infty)$ and $\kappa \in [1, \infty)$. Consider a measurable
  mapping $\varphi \colon \R^d \to \R^d$ which satisfies
  \begin{align*}
    |\varphi(x) | \le L \big(1 + |x|^\kappa \big)
  \end{align*}
  for all $x \in \R^d$. For some $p \in (2,\infty)$ let $Y \in
  L^{p\kappa}(\Omega;\R^d)$. Then there exists a constant $C$ only depending on
  $L$ and $p$ with   
  \begin{align*}
    \big\| \varphi(Y) - \varphi(Y^\circ) \big\|_{L^2(\Omega;\R^d)} \le C \big(
    1 +  \|Y\|_{L^{p\kappa}(\Omega;\R^d)}^\kappa \big)^{\frac{p}{2}}
    \delta^{\frac{1}{2} \alpha (p -2 ) \kappa}
  \end{align*}
  for all $\delta \in (0,1]$, where $Y^\circ = \min \big(1, \delta^{-\alpha}
  |Y|^{-1} \big) Y$ with arbitrary $\alpha \in (0,\infty)$.
\end{lemma}

\begin{proof}
  We apply the same idea as in the proof of \cite[Theorem~2.2]{higham2002b}.
  Consider the two measurable sets 
  \begin{align*}
    A_{\delta} := \big\{ \omega \in \Omega \; : \; | Y(\omega)| \le
    \delta^{-\alpha} \big\} \in \F
  \end{align*}
  and $A_{\delta}^c := \Omega \setminus A_{\delta}$. Note that
  $Y(\omega) = Y^{\circ}(\omega)$ for all $\omega \in A_{\delta}$. Thus,
  \begin{align*}
    \big\| \varphi(Y) - \varphi(Y^\circ) \big\|_{L^2(\Omega;\R^d)}^2 & =
    \int_{\Omega}  \big| \varphi(Y(\omega)) -   \varphi( Y^\circ(\omega))
    \big|^2 \one_{A_{\delta}^c}(\omega) \diff{\P(\omega)}. 
  \end{align*}
  For $\nu, \rho, \rho' \in (0,\infty)$ with $\frac{1}{\rho}+ \frac{1}{\rho'} =
  1$ we apply the Young inequality $ab \le \frac{\delta^\nu}{\rho} a^\rho +
  \frac{1}{\rho'} \delta^{- \nu \frac{\rho'}{\rho}} b^{\rho'}$. If we set
  $\rho = \frac{p}{2}$ then we obtain
  \begin{align*}
    &\int_{\Omega}  \big| \varphi( Y(\omega) ) - \varphi(Y^\circ(\omega))
    \big|^2 \one_{A_{\delta}^c}(\omega) \diff{\P(\omega)}\\  
    &\quad \le \frac{2 \delta^\nu}{p} \big\| \varphi( Y) - \varphi( Y^\circ)
    \big\|^{p}_{L^{p}(\Omega;\R^d)} + \big(1 - \frac{2}{p}\big) \delta^{-
    \frac{2 \nu }{p-2}} \P ( A_{\delta}^c).
  \end{align*}
  Now, the polynomial growth condition on $\varphi$ yields
  \begin{align*}
    \| \varphi(Y) - \varphi(Y^\circ) \|_{L^{p}(\Omega;\R^d)} &\le \|
    \varphi(Y)\|_{L^p(\Omega;\R^d)} + \|
    \varphi(Y^\circ)\|_{L^p(\Omega;\R^d)}\\
    &\le 2 L \big( 1 + \| Y \|_{L^{p\kappa}(\Omega;\R^d)}^\kappa \big).
  \end{align*}
  Further, it holds   
  \begin{align*}
    \P( A_{\delta}^c) = \E \big[ \one_{A_{\delta}^c} \big] \le
    \delta^{\alpha p \kappa} \E \big[ \one_{A_{\delta}^c} |Y|^{p\kappa} \big]
    \le \delta^{\alpha p \kappa} \|Y\|^{p\kappa}_{L^{p\kappa}(\Omega;\R^d)}.
  \end{align*}
  To sum up, if we choose $\nu := \alpha (p-2) \kappa$, then we obtain $\alpha p
  \kappa - \frac{2 \nu}{p-2} = \nu$ and, consequently, 
  \begin{align*}
    \big\| \varphi(Y) - \varphi(Y^\circ) \big\|_{L^2(\Omega;\R^d)}^2 &\le
    \frac{2}{p} (2L)^p \delta^{\alpha (p-2)\kappa} \big( 1 +
    \|Y\|^\kappa_{L^{p\kappa}(\Omega;\R^d)} \big)^p\\
    &\quad + \big(1 - \frac{2}{p}\big) \delta^{\alpha (p-2)\kappa}
    \|Y\|^{p\kappa}_{L^{p\kappa}(\Omega;\R^d)}. 
  \end{align*}
  This completes the proof.
\end{proof}

\begin{theorem}
  \label{th:PEMcons}
  Let $f$ and $g^r$, $r = 1,\ldots,m$, satisfy
  Assumption~\ref{as:fg} with $L \in (0,\infty)$ and $q \in (1,\infty)$. Let
  $\overline{h} \in (0,1]$ be arbitrary.
  If the exact solution $X$ to \eqref{sode} satisfies $\sup_{\tau \in [0,T]} \|
  X(\tau) \|_{L^{6q-4}(\Omega;\R^d)} < \infty$, then the projected Euler
  method $(\Psi^{\mathrm{PEM}},\overline{h},X_0)$ with $\alpha =
  \frac{1}{2(q-1)}$ is stochastically 
  B-consistent of order $\gamma = \frac{1}{2}$.  
\end{theorem}

\begin{proof}
  Let $(t,\delta) \in \mathbb{T}$ be arbitrary. First we insert \eqref{exact}
  and \eqref{eq:PsiPEM} and obtain in the same way as in the proof of
  Theorem~\ref{th:SSBEcons}
  \begin{align*}
    &X(t + \delta) - \Psi^{\mathrm{PEM}} (X(t),t,\delta)
    = \int_{t}^{t + \delta} \big( f(\tau,X(\tau)) - f(t,X(t))
    \big) \diff{\tau}\\
    &\qquad + X(t) + \delta f(t,X(t)) - X^\circ(t) - \delta f(t,X^{\circ}(t))\\
    &\qquad + \sum_{r = 1}^m \int_{t}^{t + \delta} \big( g^r(\tau,X(\tau)) -
    g^r(t,X(t)) \big) \diff{W^r(\tau)}\\
    &\qquad +\sum_{r = 1}^m \big( g^r(t,X(t)) -
    g^r(t,X^{\circ}(t) ) \big) \Delta_\delta W^r(t),
  \end{align*}
  where as before $X^\circ(t) = \min(1, \delta^{-\alpha} |X(t)|^{-1} ) X(t)$.
  In order to show \eqref{eq:cons_cond1} we therefore have to estimate
  \begin{align}
    \label{eq:term10}
    \begin{split}
      &\big\| \E \big[ X(t + \delta) -\Psi^{\mathrm{PEM}} (X(t),t,\delta) |
      \F_{t} \big] \big\|_{L^2(\Omega;\R^d)}\\
      &\quad \le \int_{t}^{t + \delta} \big\| \E \big[ f(\tau,X(\tau)) -
      f(t ,X(t)) | \F_{t} \big] \big\|_{L^2(\Omega;\R^d)}
      \diff{\tau}\\
      &\qquad + \big\|X(t)  - X^\circ(t)\big\|_{L^2(\Omega;\R^d)}
      + \delta \big\| f(t,X(t)) - f(t,X^{\circ}(t)) \big\|_{L^2(\Omega;\R^d)}.
    \end{split}
  \end{align}
  From Lemma~\ref{lem:cons1} and the inequality
  $\| \E[ Y | \F_t ] \|_{L^2(\Omega;\R^d)} \le \| Y
  \|_{L^2(\Omega;\R^d)}$ for all $Y \in L^2(\Omega;\R^d)$ we infer 
   that   
  \begin{align*}
    \int_{t}^{t + \delta} \big\| \E \big[ f(\tau,X(\tau)) - f(t,X(t))
    | \F_{t} \big] \big\|_{L^2(\Omega;\R^d)} \diff{\tau}
    \le C_{\mathrm{cons}} \delta^{\frac{3}{2}}
  \end{align*}
  for a constant $C_{\mathrm{cons}}$ depending on $L$, $q$, $m$, and
  $\sup_{\tau \in [0,T]} \|X(\tau)\|_{L^{4q-2}(\Omega;\R^d)}$.

  For the proof of \eqref{eq:cons_cond1} it therefore remains to verify that
  similar estimates hold true for the second and third term in
  \eqref{eq:term10}. For this we apply Lemma~\ref{lem:PEMcons1} with $\varphi =
  \id_{\R^d}$, $\kappa = 1$, and $p = 6q - 4$. Then we obtain
  \begin{align*}
    \big\|X(t)  - X^\circ(t)\big\|_{L^2(\Omega;\R^d)} \le C \big(1 + \| X(t)
    \|_{L^{6q - 4}(\Omega;\R^d)} \big)^{3q - 2} \delta^{\frac{3}{2}},
  \end{align*}
  since $\frac{1}{2}\alpha(p-2) = \frac{3}{2}$. A further application of
  Lemma~\ref{lem:PEMcons1} with $\varphi = f(t,\cdot)$, $\kappa = q$, and $p =
  4 - \frac{2}{q}$ yields 
  \begin{align*}
    \big\| f(t,X(t)) - f(t,X^{\circ}(t)) \big\|_{L^2(\Omega;\R^d)} \le 
    C \big(1 + \| X(t) \|_{L^{4q - 2}(\Omega;\R^d)}^q \big)^{2 -
    \frac{1}{q}} \delta^{\frac{1}{2}},
  \end{align*}
  since in this case $\frac{1}{2}\alpha(p-2) q = \frac{1}{2}$.
  Altogether, this proves \eqref{eq:cons_cond1} with $\gamma =
  \frac{1}{2}$.
  
  For the proof of \eqref{eq:cons_cond2} we have to estimate the following
  three terms 
  \begin{align}
    \label{eq:term11}
    \begin{split}
      &\big\| \big( \id - \E [ \, \cdot \, | \F_{t} ] \big)
      \big( X(t + \delta) - \Psi^{\mathrm{PEM}}(X(t),t,\delta) \big)
      \big\|_{L^2(\Omega;\R^d)} \\
      &\quad \le \int_{t}^{t + \delta} \big\| \big( \id - \E [ \, \cdot
      \, | \F_{t} ] \big) \big( f(\tau,X(\tau))  - f(t, X(t)) \big) 
      \big\|_{L^2(\Omega;\R^d)} \diff{\tau}\\
      &\qquad + \Big\| \sum_{r = 1}^m \int_{t}^{t + \delta} \big(
      g^r(\tau,X(\tau)) - g^r(t,X(t)) \big) \diff{W^r(\tau)}
      \Big\|_{L^2(\Omega;\R^d)}\\
      &\qquad + \Big\| \sum_{r = 1}^m \big( g^r(t,X(t)) -
      g^r(t,X^\circ(t)) \big) \Delta_\delta W^r(t) \Big\|_{L^2(\Omega;\R^d)}.
    \end{split}
  \end{align}
  First, we use the inequality $\| ( \id - \E[\,\cdot\,  |
  \F_t ] ) Y \|_{L^2(\Omega;\R^d)} \le \| Y \|_{L^2(\Omega;\R^d)}$ for all
  $Y \in L^2(\Omega;\R^d)$ and then we obtain from
   Lemma~\ref{lem:cons1} that
  \begin{align*}
    &\int_{t}^{t + \delta} \big\| \big( \id - \E [ \, \cdot
    \, | \F_{t} ] \big) \big( f(\tau,X(\tau)) - f(t,X(t)) \big)
    \big\|_{L^2(\Omega;\R^d)} \diff{\tau}\le C_{\mathrm{cons}}
    \delta^{\frac{3}{2}}.
  \end{align*}
  Next, we directly apply Lemma~\ref{lem:cons2} to the second term in
  \eqref{eq:term11}. This yields 
  \begin{align*}
    &\Big\| \sum_{r = 1}^m \int_{t}^{t + \delta} \big(
    g^r(\tau,X(\tau)) - g^r(t,X(t)) \big) \diff{W^r(\tau)}
    \Big\|_{L^2(\Omega;\R^d)} \le C_{\mathrm{cons}} \delta.
  \end{align*}
  Regarding the last term in \eqref{eq:term11} we obtain from 
  the It\=o isometry that
  \begin{align*}
    &\Big\| \sum_{r = 1}^m \big( g^r(t,X(t)) -
    g^r(t,X^\circ(t)) \big) \Delta_\delta W^r(t) \Big\|_{L^2(\Omega;\R^d)}^2\\
    &\quad = \delta \sum_{r = 1}^m \big\| g^r(t,X(t)) -
    g^r(t,X^\circ(t)) \big\|^2_{L^2(\Omega;\R^d)}.
  \end{align*}
  Similarly as above the estimate is completed by 
  a further application of Lemma~\ref{lem:PEMcons1} with $\varphi =
  g^r(t,\cdot)$, $\kappa = q$, and $p = 4 - \frac{2}{q}$, which gives
  \begin{align*}
    \big\| g^r(t,X(t)) - g^r(t,X^{\circ}(t)) \big\|_{L^2(\Omega;\R^d)} \le 
    C \big(1 + \| X(t) \|_{L^{4q - 2}(\Omega;\R^d)}^q \big)^{2 -
    \frac{1}{q}} \delta^{\frac{1}{2}}.
  \end{align*}
  Thus, as desired it holds
  \begin{align*}
    &\Big\| \sum_{r = 1}^m \big( g^r(t,X(t)) -
    g^r(t,X^\circ(t)) \big) \Delta_\delta W^r(t) \Big\|_{L^2(\Omega;\R^d)}
    \le  C_{\mathrm{cons}} \delta,
  \end{align*}
  for a constant $C_{\mathrm{cons}}$ depending on $L$, $q$, $m$, and
  $\sup_{\tau \in [0,T]} \| X(\tau) \|_{L^{4q - 2}(\Omega;\R^d)}$.
\end{proof}

We conclude this section by stating the strong convergence result for the PEM
method, which follows directly from Theorems~\ref{th:PEMstab} and
\ref{th:PEMcons} as well as Theorem~\ref{th:Bconv}. 

\begin{theorem}
  \label{th:PEMconv}
  Let $f$ and $g^r$, $r = 1,\ldots,m$, satisfy
  Assumption~\ref{as:fg} with $L \in (0,
  \infty)$, $\eta \in (\frac{1}{2},\infty)$,
  and $q \in (1,\infty)$. Let $\overline{h} \in
  (0,1]$. If the exact solution $X$ to \eqref{sode}
  satisfies $\sup_{\tau \in [0,T]} \| X(\tau) \|_{L^{6q-4}(\Omega;\R^d)} <
  \infty$, then the projected Euler-Maruyama method
  $(\Psi^{\mathrm{PEM}},\overline{h},X_0)$ with $\alpha = \frac{1}{2(q-1)}$ is
  strongly convergent of order $\gamma = \frac{1}{2}$. 
\end{theorem}

\section{Numerical experiments}
\label{sec:exp}
In this section we perform a series of numerical experiments which aim to
illustrate the strong convergence results of the previous sections. In
particular, we compute estimates of the strong error of convergence for the
numerical discretization of the stochastic Ginzburg-Landau equation
\cite[Chap.~4.4]{kloeden1999} and the $3/2$-stochastic volatility model (see 
e.g. \cite{goard2013, henryL2007} and \cite[Sec.~1]{sabanis2013b}). 

First, we consider the stochastic Ginzburg-Landau equation (GLE) given by 
\begin{align}
  \begin{split}
    \label{eq:Ginz_landau}
    \diff X(t)&= \big(-X^3(t)+(\mu+\frac{1}{2}\sigma^2)X(t) \big)\diff t+\sigma
    X(t)\diff W(t),  \\
    X(0)&=X_0,
  \end{split}
\end{align}
where $\mu$, $\sigma$, $t\geq 0$. 
This equation satisfies
Assumption~\ref{as:fg} and condition \eqref{eq:growthcond} with $q=3$
 since the cubic term in the drift function has a negative sign. 
As already noted in
\cite[Chap.~4.4]{kloeden1999} the exact solution to \eqref{eq:Ginz_landau} 
is
\begin{align}
  \label{eq:GLex} 
  X(t)=X_0\exp(\mu t+\sigma W(t))\big(1+2X^2_0\int_0^t\exp(2\mu  
  s+2\sigma W(s))\diff s\big)^{-\frac{1}{2}},\quad t \ge 0.
\end{align}
Having an explicit expression for the exact solution, explains why
the GLE is often used for numerical experiments
in the literature. For instance, we refer to \cite{wang2012}, where similar
experiments have been conducted for split-step one-leg theta methods. 

\begin{figure}[H]
\includegraphics[width=11cm]{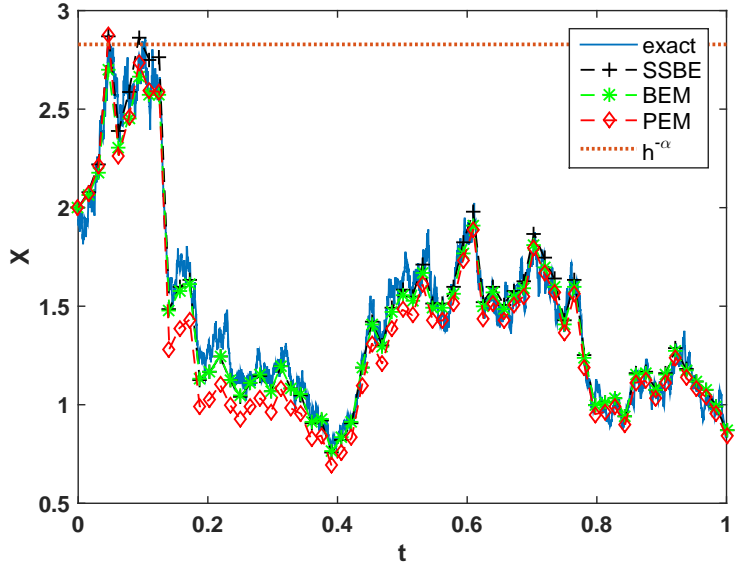}
\caption{Single trajectories of each numerical method with step size $h=2^{-6}$ 
of the stochastic Ginzburg-Landau equation with parameters $\mu=0.5$,
$\sigma=1$  and $X_0=2$. Threshold for the PEM projection is a shown as a
dotted line.} 
\label{fig}
\end{figure}

In our experiments the SODE \eqref{eq:Ginz_landau} is discretized by the
split-step backward Euler method, the backward Euler-Maruyama scheme and the
projected Euler-Maruyama method, respectively. Figure~\ref{fig} shows single
trajectories of the exact solution and the three numerical methods with
equidistant step size $h = 2^{-6}$ and parameter values $\mu=0.5$, $\sigma=1$,
and $X_0=2$. Since Assumption~\ref{as:fg} is satisfied
with growth rate $q = 3$, the parameter value $\alpha= \frac{1}{2(q-1)} =
\frac{1}{4}$ is used for the PEM method.
The implementation of the two implicit schemes SSBE and BEM employs
Cardano's method for directly solving the nonlinear equations.
Further, for the simulation of the exact solution it is necessary to approximate
the deterministic integral appearing in \eqref{eq:GLex}. This is done 
by a Riemann sum with step size $2^{-12}$.

Regarding the PEM method we are particularly interested in such trajectories
which do not coincide with those generated by the standard Euler-Maruyama
method. This event occurs when the scheme leaves the sphere of radius 
$h^{-\alpha}$ at least once and is then drawn back by the projection.
More precisely, if it holds true that
\begin{align}
  \label{eq:proj}
  \big\{ i= 1,\ldots, N \; : \; |X_h^{\mathrm{PEM}}(t_i)| > h^{-\alpha} \big\}
  \neq \emptyset, 
\end{align}
then the PEM method deviates from the standard Euler-Maruyama scheme.
In Figure~\ref{fig} the trajectory 
of the PEM method crosses the line with height $h^{-\alpha} =
2^{\frac{3}{2}}$ in the fourth step. Thereafter, the scheme seemingly
underestimates the exact solution, although
this effect vanishes when time evolves due to the dissipative nature of
equation  \eqref{eq:Ginz_landau}.  

Obviously, this behavior is undesirable and the standard
Euler-Maruyama method would have given a better approximation of the exact
solution in this case. However, let us stress that the main purpose of the
projection in the PEM method is to 
counteract the effect described by Hutzenthaler et al.~\cite{hutzenthaler2011},
where the product of the norm of explosive trajectories by the explicit
Euler-Maruyama method times the probability to observe such explosive
trajectories goes to infinity as the step size goes to zero. Thus, the
Euler-Maruyama values have no bounded moments in the limit $h \to 0$ and,
consequently, it is divergent in the mean square sense. 

On the other hand, the projection in the PEM method essentially prevents the
numerical methods from leaving the ball with radius $h^{-\alpha}$. Due to the
existence of higher moments, this also holds true for the exact solution up to
a set of very small probability (c.f. with the proof of
Lemma~\ref{lem:PEMcons1}). Hence, while being possibly large, the error in such
an instance remains essentially bounded and the line of arguments in
\cite{hutzenthaler2011} leading to the divergence of the standard
Euler-Maruyama method does not apply to the PEM method.  

\begin{figure}[t]
\includegraphics[width=11cm]{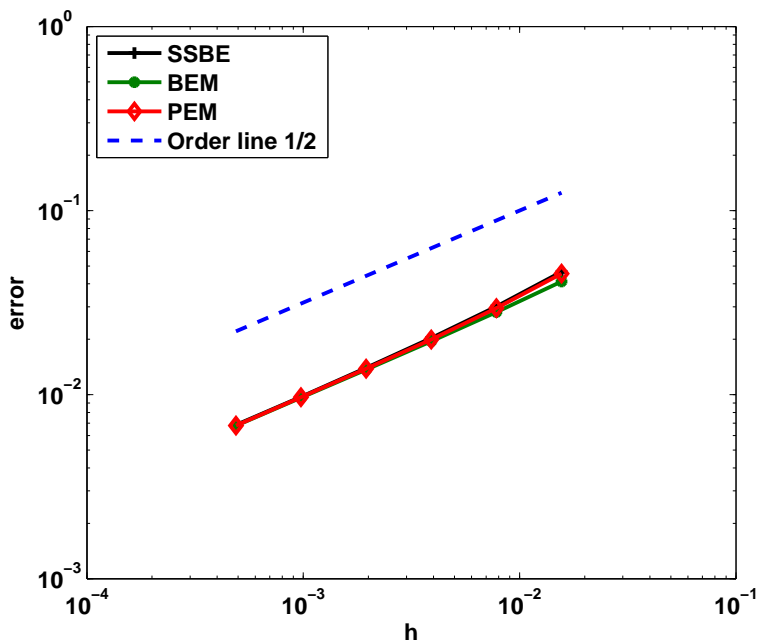}
\caption{Strong convergence errors for the approximation of the stochastic
Ginzburg-Landau equation \eqref{eq:Ginz_landau} with parameters $\mu=0.5$, 
$\sigma=1$, and $X_0=2$ for .}  
\label{fig1}
\end{figure}

\begin{table}
  \caption{Estimated errors and EOCs for the approximations of
  \eqref{eq:Ginz_landau}}
  \label{tab:matprog}
  \begin{tabular}{p{1.1cm}p{1.3cm}p{1.2cm}p{1.3cm}p{1.2cm}p{1.3cm}p{0.8cm}p{1.2cm}}
       &  SSBE  &  &  BEM   & &          PEM  &  
\\   \noalign{\smallskip}\hline\noalign{\smallskip}
  $h$   &   error  &  EOC   &   error   &  EOC  &  error  &  EOC &           \#-Proj.
\\   \noalign{\smallskip}\hline\noalign{\smallskip}
       $2^{-6}$ &      0.04637  &    &      0.04106 &    &       0.04553  &   &    33906\\
     $2^{-7}$ &      0.03013  &          0.62  &      0.02808 &          0.55  &       0.02945  &         0.63 &     2157\\
     $2^{-8}$ &      0.02029  &          0.57  &      0.01951 &          0.53  &       0.02002  &         0.56 &       26\\
     $2^{-9}$ &      0.01396  &          0.54  &      0.01365 &          0.52  &       0.01384  &         0.53 &        0\\
     $2^{-10}$ &      0.00975  &          0.52  &      0.00960 &          0.51  &       0.00968  &         0.52 &        0\\
     $2^{-11}$&      0.00683  &          0.51  &      0.00678 &          0.50  &       0.00681  &         0.51 &        0
  \end{tabular}
\end{table}

Table~\ref{tab:matprog} and
Figure~\ref{fig1} show the estimated strong error of
convergence for six different equidistant step sizes $h= 2^{k - 12}$,
$k=1,\dots,6$. For simplicity we only estimate the error at the final time $T
= 1$, that is  
\begin{align}\label{eq:error_def}
 \text{error}=(\E(|X_h(T)-X(T)|^2))^\frac{1}{2},
\end{align}
where $X_h(T)$ denotes the respective numerical approximations of the exact 
solution $X(T)$. The
expected value is estimated by a Monte Carlo simulation based on $10^6$
sample paths. Our experiments indicate that the Monte Carlo error
then drops well below the strong error to be estimated. As before the 
parameter values are $\mu=0.5$, $\sigma=1$, and $X_0=2$. 

In Figure~\ref{fig1} one clearly observes strong order $\gamma=\frac{1}{2}$ for
all three methods. Further, 
no numerical method has a significant advantage over one of the others.
Table~\ref{tab:matprog} also contains the estimates of the errors
and the corresponding \emph{experimental order of convergence} defined by 
\begin{align*} 
\text{EOC}=\frac{\log(\text{error}(h_i))-\log(\text{error}(h_{i-1}))}{
\log(h_i)-\log(h_ { i-1 } ) }, \ i=2,\dots,k.
\end{align*}
For each method we also computed an average of the experimental order of
convergence by determining the best fitting line in a least-squares sense for
the logarithmically scaled errors. The slopes of these lines are $0.55$,
$0.52$, and $0.54$ for the SSBE, BEM, and PEM method, respectively.

Finally, the last column in Table~\ref{tab:matprog} contains the number of
Monte Carlo samples for which the trajectory of the PEM method leaves the
sphere of radius $h^{-\alpha}$, that is the event described in \eqref{eq:proj}
has occurred. Relating this to the total number of Monte Carlo samples we see
that approximately $3.3$ percent of the PEM trajectories with step size $h=
2^{-6}$ do not coincide with the trajectories of the standard Euler-Maruyama
method. However, as the step size gets smaller the number of those samples
drops quickly. 

Note that we do not know if those excursions from the sphere of radius
$h^{-\alpha}$ are caused by the explosive behavior of the standard
Euler-Maruyama method described in \cite{hutzenthaler2011} or if it is due to
an intrinsic feature of the exact solution. In the latter case the projection
may cause more harm than good. But in both cases a good advice is
to choose a smaller step size if the relative frequency to observe the event
\eqref{eq:proj} is too high.

This becomes even more evident in our next example, which consists of the
following nonlinear SODE
\begin{align}
 \begin{split}
  \label{eq:volatility}
    \diff X(t)&=\lambda X(t)(\mu-|X(t)|)\diff t+\sigma|X(t)|^{\frac{3}{2}}\diff
    W(t),  \\ 
    X(0)&=X_0,
 \end{split}
\end{align}
where $\lambda$, $\mu$, $\sigma$, $X_0 \geq 0$. This equation incorporates a 
super-linearly growing diffusion coefficient function and is used as 
a stochastic volatility model (SVM) in mathematical finance \cite{goard2013}.
It has also been considered in \cite{sabanis2013b} for a tamed Euler method. 

The mappings $f,  \ g\colon\R \rightarrow \R$ defined by $f(x):=\lambda 
x(\mu-|x|)$ and $g(x):=\sigma|x|^\frac{3}{2}$ are continuous for all $x\in\R$ 
and satisfy the global monotonicity condition in Assumption \ref{as:fg} with 
$\eta\leq\frac{\lambda + \sigma^2 }{\sigma^2}$ and $L=\lambda\mu$. Moreover,
the coercivity condition \eqref{eq:growthcond} is fulfilled for every $p \le
\frac{2 \lambda + \sigma^2}{\sigma^2}$.
We refer to the Appendix in \cite{sabanis2013b} for calculations 
of the constants $\eta$, $p$, and $L$.

For the numerical experiments the parameter values are
$\lambda=3.5$, $\mu=3$, $\sigma=1$, and the initial value is $X_0=5$. 
Hence, the global monotonicity condition \eqref{eq3:onesided} is satisfied with
$1 < \eta < 4.5$. Further, the exact solution fulfills $\sup_{t\in [0,T]} \|
X(t) \|_{L^p(\Omega;\R^d)} < \infty$ for every $p \le 8$. 

\begin{figure}[t]
\includegraphics[width=11cm]{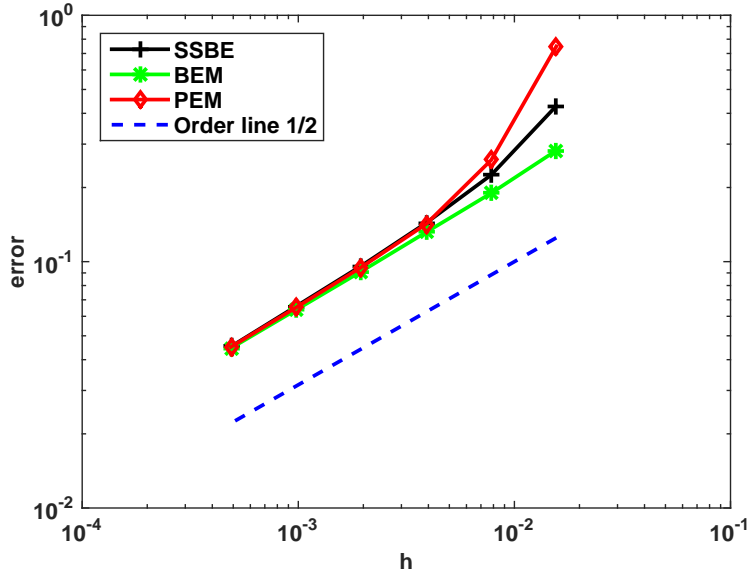}
\caption{Strong convergence errors for the approximation of the
$3/2$-volatility model \eqref{eq:volatility} with parameters $\lambda=3.5,
\mu=3, \sigma=1$ and $X_0=5$.}
\label{fig2}
\end{figure}

\begin{table}
  \caption{Estimated errors and EOCs for the approximation of
  \eqref{eq:volatility}}
  \label{tab:exp}
  \begin{tabular}{p{1.1cm}p{1.3cm}p{1.2cm}p{1.3cm}p{1.2cm}p{1.3cm}p{0.8cm}p{1.2cm}}
       &  SSBE  &  &  BEM   & &          PEM  &  
\\   \noalign{\smallskip}\hline\noalign{\smallskip}
  $h$   &   error  &  EOC   &   error   &  EOC  &  error  &  EOC &  \#-Proj. \\   \noalign{\smallskip}\hline\noalign{\smallskip}
  $2^{-6}$ &      0.42856  &  &   0.28267 &    &       0.74770  & &   139890\\
  $2^{-7}$ &      0.22499  &          0.93  &      0.18973 &          0.58  &       0.25858  &         1.53 &    22338\\
  $2^{-8}$ &      0.14359  &          0.65  &      0.13167 &          0.53  &       0.14208  &         0.86 &     2707\\
  $2^{-9}$ &      0.09603  &          0.58  &      0.09119 &          0.53  &       0.09484  &         0.58 &      294\\
  $2^{-10}$ &      0.06585  &          0.54  &      0.06371 &          0.52  &       0.06516  &         0.54 &       24\\
  $2^{-11}$ &      0.04524  &          0.54  &      0.04454 &          0.52  &
  0.04508  &         0.53 &        0 
  \end{tabular}
\end{table}

Since there is no explicit expression available, we replace the exact solution
in \eqref{eq:error_def} by a numerical reference approximation with a very fine
step size $h_{\mathrm{ref}} = 2^{-16}$. The implicit schemes are again
implemented by solving the nonlinear equation in each time step explicitly.
This time we take the parameter value $\alpha=\frac{1}{2}$ for the PEM method.
As above our estimate of the errors are based on a Monte Carlo simulation with
$10^6$ sample paths. 
 
Figure~\ref{fig2} shows the strong convergence errors of the three methods with
six different step sizes $h=2^{k}\Delta t$, $k=5,\dots,10$. The results
are well in line with the predicted strong order $\gamma=\frac{1}{2}$ for all
schemes provided that the step size is sufficiently small. In that case, there
is again no significant difference in the behavior of the three schemes. For
larger step sizes, however, the BEM methods outperforms the
SSBE scheme and, on a much larger scale, the
PEM method significantly.

This can also been seen from Table~\ref{tab:exp}, which contains
the numerical values for the strong errors shown in Figure~\ref{fig2}. The
values for the corresponding experimental order of convergence verify the
theoretical results only for small values of $h$. As above we also determine an
average experimental order of convergence for the three methods as the slope of
the best fitting line in the mean-square sense. The results for the SSBE, BEM,
and PEM method are $0.63$, $0.53$, and $0.77$, respectively. 

Note that in the stochastic volatility method the magnitude of the noise term
is, intentionally, much larger than in the Ginzburg-Landau equation while the
damping in the drift term is weaker if $X(t) > \mu$. Thus, the dynamic is more
often dominated by the noise term. It appears that the BEM scheme works best in
this situation as the noise term is always damped by the implicit step. In the
SSBE scheme, on the other hand, the most recent noise increment is undamped
which apparently affects the error negatively if the step size is large. 
This effect is even worse for the explicit PEM scheme. In addition, the high
noise intensity makes it more likely for the exact solution to leave the sphere
of radius $h^{-\alpha}$ while the PEM method cannot follow and is pulled back.
This coincides with a larger number of trajectories in which the projection has
been applied as can be seen from the values in the last column of
Table~\ref{tab:exp}. 

To conclude this section, let us summarize our observations: We have seen in
the numerical experiments that the three schemes perform equally well if the
step size is small enough. For larger step sizes the implicit schemes
turned out to be superior over the PEM method, especially if 
the noise term is more likely to dominate the underlying dynamics. However, by
observing the relative frequency of the event \eqref{eq:proj} one may have a
simple indicator available if the step size of the explicit method should be
further decreased. Since the PEM method is, in general, cheaper to
simulate than the implicit schemes, one might afford this, eventually.

\subsection*{Acknowledgement}

The authors wish to thank R.~D.~Grigorieff for calling our attention
to the reference \cite{strehmel2012} and thereby pointing us to the concept of
C-stability. Further, the first two authors gratefully acknowledge financial
support by the DFG-funded CRC 701 'Spectral Structures and Topological Methods
in Mathematics'. The same holds true for the third named author, who has been
supported by the research center \textsc{Matheon}. He also likes to
thank the CRC 701 for making possible a very fruitful research stay at
Bielefeld University, during which essential parts of this work were written.
Finally, the authors wish to thank an anonymous referee and the associated
editor for very helpful suggestions and comments.

\def\cprime{$'$} \def\polhk#1{\setbox0=\hbox{#1}{\ooalign{\hidewidth
  \lower1.5ex\hbox{`}\hidewidth\crcr\unhbox0}}}

\end{document}